\documentclass[11pt]{article}

\usepackage[english]{babel}
\usepackage{amsmath,amssymb,amsfonts,amsthm,mathrsfs,bbm}
\usepackage{graphicx,xcolor}
\usepackage{enumerate,geometry,marginnote,afterpage}
\usepackage{fancyhdr,courier,type1cm,dsfont,float,tikz}
\usepackage{hyperref}
\usetikzlibrary{arrows, positioning}

\numberwithin{equation}{section}

\theoremstyle{plain} 

\newtheorem{theorem}{Theorem}[section]
\newtheorem{corollary}[theorem]{Corollary}
\newtheorem{lemma}[theorem]{Lemma}

\newtheorem{definition}[theorem]{Definition}

\newtheorem{remark}[theorem]{Remark}

\newcommand{\Z}{\mathbb Z}

\newcommand{\N}{\mathbb N}
\newcommand{\R}{\mathbb R}

\newcommand{\dd}{\mathrm{d}}
\newcommand{\ee}{\mathrm{e}}

\newcommand{\EE}{\mathbb E}
\newcommand{\PP}{\mathbb{P}}
\newcommand{\Tt}{\mathcal{T}}

\newcommand{\tX}{\tilde{X}}

\newcommand{\cD}{\mathcal{D}}

\makeatletter
\renewcommand{\title}[1]{\gdef\@title{\LARGE #1}}
\makeatother

\begin{document}

\title{Large deviation principle for friendship-biases\\ 
in Galton--Watson trees}

\author{
\renewcommand{\thefootnote}{\arabic{footnote}}
Frank den Hollander
\footnotemark[1]
\\
\renewcommand{\thefootnote}{\arabic{footnote}}
Azadeh Parvaneh
\footnotemark[2]
}

\footnotetext[1]{
Mathematical Institute, Leiden University, Einsteinweg 55, 2333 CC Leiden, The Netherlands.\\
\emph{email}: denholla@math.leidenuniv.nl
}

\footnotetext[2]{
Faculty of Technology, Bielefeld University, 33501 Bielefeld, Germany.\\
\emph{email}: azadeh.parvaneh@uni-bielefeld.de
}

\maketitle

\begin{abstract}
In this paper we consider the friendship-bias of the vertices in an infinite rooted Galton--Watson tree. The friendship-bias of a vertex is the difference between the average degree of the neighbours of the vertex and the degree of the vertex itself. A vertex is said to be of type $\chi \in S$, with $S = \{-,0,+\}$, when its friendship-bias is, respectively, strictly negative, zero or strictly positive. We consider the fractions $f_l^\chi$ of vertices of type $\chi \in S$ along a random downward path up to branching depth $l \in \N$ and derive a large deviation principle (LDP) for the triple $(f_l^\chi)_{\chi \in S}$ as $l\to\infty$. The branching depth of a vertex counts the number of branchings that occur along the path that connects the vertex to the root of the tree. The rate in the LDP is $l$, while the rate function in the LDP is identified in terms of a variational formula minimising a relative entropy under a linear constraint. 

We focus on the case of binary branching, for which the rate function is already quite involved. We identify the qualitative properties of the rate function and show how it can be computed numerically. We briefly indicate how to proceed for more general branching and for vertex types along a tree consisting of a finite number of random downward paths. Our paper is the first to consider large deviations of vertex types. 

\medskip\noindent
{\it AMS} 2020 {\it subject classifications.}
60F10, 
60J80, 
60C05. 


\medskip\noindent
{\it Acknowledgment.} The work in this paper was supported by the Netherlands Organisation for Scientific Research (NWO) through Gravitation-grant NETWORKS-024.002.003. FdH was supported by the National Science Foundation (NSF) under Grant No.\ DMS-1928930 while in residence at the Simons Laufer Mathematical Sciences Institute in Berkeley, California, USA during the Spring 2025 semester. AP received funding from the European Union's Horizon 2020 research and innovation programme under the Marie Sk\l odowska-Curie grant agreement Grant Agreement No.\ 101034253. The authors thank Rajat Hazra and Nelly Litvak for fruitful discussions.

\vspace{0.1cm}
\hfill\includegraphics[scale=0.1]{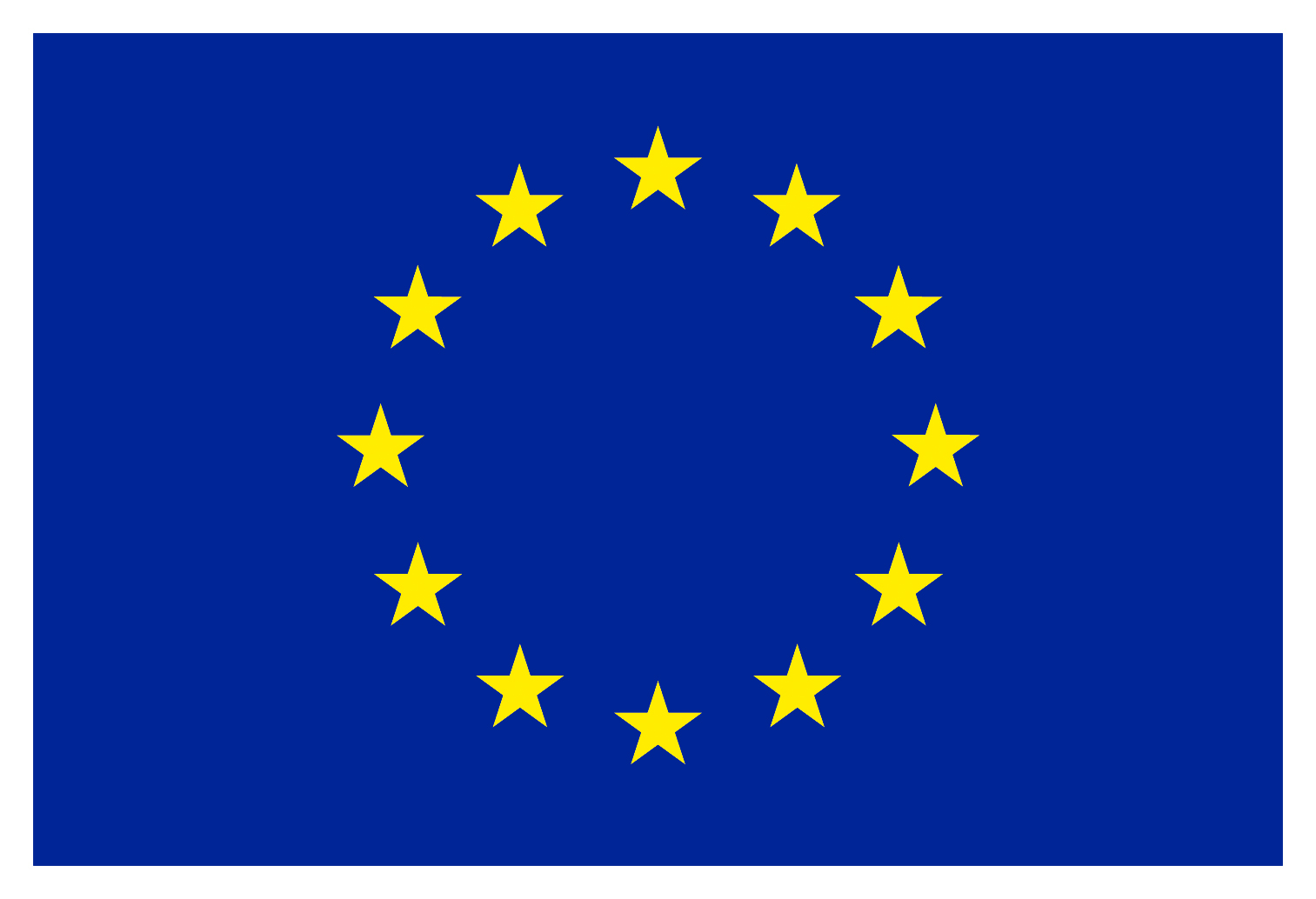}

\end{abstract}


\newpage 

\small
\tableofcontents
\normalsize


\section{Introduction}

In this paper we consider the friendship-bias of the vertices in an infinite rooted Galton--Watson tree. Our goal is to derive an LDP for the fractions of vertices in a random downward path whose friendship-bias is, respectively, strictly negative, zero or strictly positive. Our further goal is to identify the qualitative properties of the associated rate function, and show how the rate function can be computed numerically. Our main focus will be on binary branching, but we do indicate how to deal with more general offspring laws of the Galton--Watson tree and with more general random downward paths. 


\subsection{Friendship-bias}

Let $G_n$ be a \emph{finite undirected graph} with $n$ vertices, labelled by $u \in [n] = \{1,\ldots,n\}$. Define the friendship-bias of vertex $u$ as
\[
\Delta_{u,n} = \left(\frac{1}{d_u} \sum_{v \in [n]} A_{uv} d_v - d_u\right) \,\mathbbm{1}_{\{d_u \neq 0\}}, 
\]
where $A=(A_{uv})_{u,v\in [n]}$ is the adjacency matrix of $G_n$, i.e., $A_{uv}$ is the number of edges between $u \neq v$ and $A_{uu}$ is twice the number of self-loops at $u$, and $d_u = \sum_{v \in [n]} A_{uv}$ is the degree of $u$. Then the friendship paradox says that
\[
\frac{1}{n} \sum_{u \in [n]} \Delta_{u,n} \geq 0,
\] 
with equality if and only if all connected components of $G_n$ are regular. Typically, $G_n$ contains vertices of all three types, i.e., $u$ for which $\Delta_{u,n}$ is $>0$, $=0$, $<0$. Even though the average is non-negative, there are graphs for which the number of vertices $u$ with $\Delta_{u,n} < 0$ is larger than the number of vertices $u$ with $\Delta_{u,n} > 0$. The goal of the present paper is to derive an LDP for the fractions of such vertices in \emph{infinite Galton--Watson trees}. For background on the friendship paradox and references to the relevant literature, we refer to \cite{HHP1,HHP2,HHLP}. Mathematically, the friendship paradox has so far received only modest attention, and a proper \emph{quantification} of friendship-biases in large-size networks is still lacking.  


\subsection{Earlier work and present paper}


\paragraph{Large deviations on random trees.}

Large deviations for functions of random trees are technically challenging because of the underlying branching structure. An example is the work in \cite{DGPZ}, which establishes large deviation principles under the quenched law and the annealed law for the scaled height of a random walk on a supercritical Galton--Watson tree and shows that the associated rate functions coincide almost surely. The analysis crucially exploits the branching structure of the tree, together with a version of the ubiquity lemma for Galton--Watson trees introduced in \cite{GK}, which roughly states that any non-rare local property almost surely occurs in subtrees rooted at a positive fraction of the vertices along every ray of the tree (see \cite[Lemma 2.2]{DGPZ}).
 
In \cite{DMS}, a large deviation principle is proved for the empirical subtree measures of multitype Galton--Watson trees, where an empirical subtree measure records the law of rooted subtrees obtained by viewing the process from each vertex. For this purpose, a tree-adapted notion of shift-invariance is introduced for probability measures on trees, with a shift given by rerooting of the tree at a child and passing to the corresponding progeny subtree. The study in \cite{BMS} establishes a large deviation principles for the number of leaves in preferential attachment graphs and random recursive trees, as well as for the number of cherries in Yule trees.

In \cite{BB1,BB2}, large deviation principles are proved for population sizes in branching processes in random environments, viewed as a generalisation of Galton--Watson processes in which the offspring law is randomly selected at each generation. The first paper analyses both downward and upward large deviations and derives the lower-deviation bound by exhibiting a scenario in which the branching process evolves typically under an atypical sequence of environments. The second paper investigates upward large deviations in the presence of possibly heavy-tailed offspring laws and derives an upper-deviation bound also for Galton--Watson processes with heavy-tailed offspring laws.


\paragraph{The friendship paradox on random trees.}

In \cite{HHP1}, a quantification of the friendship paradox is carried out for sparse locally tree-like random graphs, and introduces the notion of \emph{significance} of the friendship paradox. Specifically, the friendship paradox is said to be asymptotically significant for a sequence of finite random graphs $\{G_n\}_{n\in\N}$ labelled by the number of vertices $n$ if the proportion of vertices with non-negative friendship-bias converges in probability as $n\to\infty$ to a number $\geq \tfrac{1}{2}$. Interestingly, we found that the friendship paradox is significant for the homogeneous Erd\H{o}s--R\'enyi random graph, the inhomogeneous Erd\H{o}s--R\'enyi random graph, as well as for several classes of configuration random graphs and preferential attachment random graphs (see \cite{vdH1,vdH2} for a description).  

In \cite{HHP2}, the \emph{multi-level} friendship paradox is analysed, which addresses the friendship paradox at higher levels of friendship. More specifically, by extending the notion of friendship-bias to a multi-level setting, using random walks on graphs and imposing additional constraints on the graph, we showed that the $k$-level average friendship-bias is non-negative for each fixed $k$, regardless of whether $k$-level friends are explored via backtracking or non-backtracking random walks. 

In \cite{HHLP}, a probabilistic framework is developed for understanding and analysing the friendship paradox on \emph{trees}. We derived sharp results on the density and the correlation of the three types of vertices in tree environments that arise naturally in real-world networks.

In \cite{BGHJS}, the friendship-bias is generalised to different types of subgraphs, i.e., instead of monitoring the number of edges a vertex is part of the focus is on the number of wedges, triangles, etc. It is shown that the \emph{generalised friendship paradox} holds for wedges, but not for triangles. A detailed analysis is given of the average friendship-bias for several classes of sparse and dense random graphs. In \cite{HV}, it is shown that the generalised friendship paradox holds for a class of centrality measures, including PageRank and Betweenness. 


\paragraph{Large deviations for vertex types on random trees.}

Establishing large deviation principles for quantities reflecting the friendship-bias inherently involves handling information on degrees in three generations of the tree simultaneously. In the present paper, we consider the fractions of vertices of a given type (strictly negative, zero or strictly positive) along a random ray (i.e., a random downward path from the root) up to a given branching depth, derive a large deviation principle (LDP) for the triple of these fractions, and identify the qualitative properties of the associated rate function as a function of the offspring law of the Galton--Watson tree. We focus on the case of \emph{binary branching}, for which the rate function is already quite complex, and indicate how to proceed for more general branching.

Our paper is the first to consider large deviations of vertex types. In order to derive the LDP, a number of hurdles have to be overcome, which we deal with as we go along. In particular, the exploration is captured in terms of a certain \emph{algorithm} that allows us to express the fractions of vertex types in terms of a certain Markov chain, and the large deviations of this Markov chain have to be analysed subject to certain \emph{constraints}. Other forms of exploration, different from the exploration up to a given branching depth considered here, are possible too, such as exploration up to a given generation. However, these tend to be more difficult to handle and we will not pursue them here.  


\subsection{Outline}

The remainder of this paper is organised as follows. In Section~\ref{sec:exploration} we focus on Galton--Watson trees with a binary offspring law, i.e., each vertex has an offspring of size either one or two. We introduce a pathwise exploration of the tree that keeps track of the branching depth and the type of the vertices it encounters, where the branching depth of a vertex counts the number of branchings that occur along the path that connects the vertex to the root of the tree. We state a lemma that identifies the vertex type in terms of the exploration, and formulate a recursion that counts the number of vertices of a given type along the path up to a given branching depth. We use this lemma to state two theorems that establish an LDP for these numbers and identify the qualitative properties of the associated rate function. The proofs of the lemma and the two theorems are given in Section~\ref{sec:proofs}. 

At the end of Section~\ref{sec:exploration} we briefly indicate how to go beyond binary branching. We do not pursue this extension in detail, because the rate function for binary branching is already quite involved. At the end of Section~\ref{sec:proofs} we show how the rate function for binary branching can be computed numerically.


\section{Pathwise exploration of the tree}
\label{sec:exploration}

In this section we consider a Galton--Watson tree with a \emph{binary} offspring law. In order to control the large deviations of the vertex types, we explore the tree along a \emph{random path} that starts from the root and ends at a vertex with a given \emph{branching depth}. Section~\ref{sec:AP-not} introduces notation, Section~\ref{sec:vertextypes} characterises the vertex types, Section~\ref{sec:expl} defines the exploration, Section~\ref{sec:typecount} counts the vertex types along the exploration path, while Section~\ref{sec:LDP} states our main results: a theorem that states an LDP for the fraction of vertex types up to a given branching depth and a theorem that identifies the quantitative properties of the associated rate function. In Section \ref{sec:comparison} we compare the typical fractions for our exploration according to branching depth with those found earlier for exploration according to generation, and show that they are \emph{different}. In Section~\ref{sec:beyondbinary} we briefly indicate how to move beyond binary branching.  


\subsection{Notation}
\label{sec:AP-not}

Let $(\Tt,\phi)$ be an infinite rooted tree with vertex set $V(\Tt)$ and root vertex $\phi$. To label the vertices, use the Ulam--Harris method of labelling a tree: each vertex $u \in V(\Tt) \setminus \{\phi\}$ is labelled by a finite word $\phi\, u^1 \ldots u^k$, where $k \in \N$, and, for $u \in V(\Tt)$, the $j$th offspring of $u$ is
\[
u_{j} = uj = \left\{
\begin{array}{ll}
\phi\, j, &\text{if $u = \phi$},\\[0.2cm]
\phi\, u^{1}\ldots u^{k} j, &\text{if $u = \phi\, u^{1} \ldots u^{k}$}.
\end{array} 
\right. 
\]
Denote the parent of $u \in V(\Tt) \setminus \{\phi\}$ by $u_{-1}$. Let $X_u$ be the number of \emph{offspring} of $u$. 

For $u, v \in V(\Tt)$, write $u\sim v$ to indicate that vertex $u$ is adjacent to vertex $v$. The edge set of $\Tt$ is $E(\Tt) = \{\{u,v\}\colon\, u,v\in V(\Tt), u \sim v\}$. Let $d_u$ be the degree of a vertex $u$. Define the \textit{friendship-bias} of $u$ as
\[
\Delta_{u} = \frac{1}{d_u} \sum_{\substack{v \in V(\Tt) \\ v \sim u}} d_v - d_u,
\]
which represents the difference between the average degree of the neighbours of $u$ and the degree of $u$ itself. Throughout the paper we follow the natural convention of defining the empty sum as $0$, in particular, the friendship-bias of the root in a trivial tree is $0$. A vertex is said to be \textit{positive} when its friendship-bias is strictly positive, \textit{negative} when its friendship-bias is strictly negative, and \textit{neutral} otherwise. 

For $i\in\Z$, abbreviate $\N_i = \{i,i+1,\ldots\}$. Write $\N=\N_1$. Let $\Tt_\infty$ denote the infinite Galton--Watson tree rooted at a single ancestor $\phi$, defined on a probability space $(\Omega, \mathscr{F},\PP)$, with offspring law $(p_k)_{k\in\N_0}$ satisfying
\[
p_0=0, \qquad p_1=p\in (0,1), \qquad p_{1}+\cdots+p_{m}=1 \text{ for some } m\in\N_{2}.
\]
In what follows we restrict to $m=2$, i.e., binary branching with $p_1 = p$ and $p_2 = 1-p$. 

Write
\[
\mathcal{S}_2=\{s = (s_\chi)_{\chi \in S} \in [0,1]^3\colon\, \sum_{\chi \in S} s_\chi=1\}
\]
to denote the two-dimensional probability simplex in $\R^3$, and equip it with the topology induced by the total variation distance
\[
d_{\mathrm{TV}}(s,s') = \tfrac{1}{2}\|s-s'\|_1,
\]
which turns $\mathcal{S}_2$ into a Polish space.
For a Polish space $E$, we also write $\mathcal{P}(E)$ for the space of probability measures on $E$, endowed with the \textit{weak topology} in the sense of the topology of weak convergence of probability measures. Throughout the sequel, whenever we say that a family of random variables satisfies an LDP, we actually mean that the family of \emph{laws} of these random variable satisfies that LDP.


\subsection{Characterisation of the vertex types}
\label{sec:vertextypes}

To describe the local structure of $\Tt_\infty$, we classify its non-root vertices into three categories: branching points, linear points, and branching-linear points (see Figure~\ref{fig:tree-structure}). The root is considered separately, but it may itself fall into the class of branching points.

\begin{definition}{\bf [Vertex types]}
$\mbox{}$
\begin{itemize}
\item[{\rm (a)}] 
{\rm A vertex $u \in V(\Tt_\infty)$ is called a \textit{branching point} if $X_u = 2$.}
\item[{\rm (b)}] 
{\rm A vertex $u \in V(\Tt_\infty) \setminus \{\phi\}$ with $d_u = 2$ is called a \textit{linear point} if its parent is not a branching point, and a \textit{branching-linear point} if its parent is a branching point.}
\item[{\rm (c)}] 
{\rm A \emph{linear piece} from a vertex $v$ is the maximal path
\[
v = v_0 \to v_1 \to \cdots \to v_{\ell(v)} = \tau(v),
\]
where the terminal vertex $\tau(v)$ is the first branching point on the path, i.e. $X_{\tau(v)} = 2$. The random integer $\ell(v)\in\mathbb N_0$ is called the \emph{length} of the linear piece from $v$.}
\end{itemize}
\end{definition}

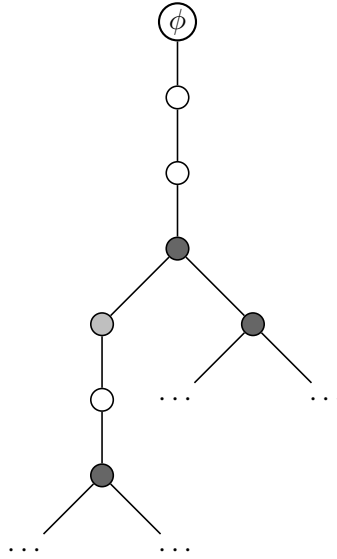
\begin{figure}[htbp]
\centering
\tikzset{
branching/.style={circle, draw=black, fill=black!60, inner sep=3pt},
linear/.style={circle, draw=black, fill=white, inner sep=3pt},
blpoint/.style={circle, draw=black, fill=gray!50, inner sep=3pt},
cont/.style={draw=none, fill=none},
root/.style={circle, draw=black, thick, fill=white, inner sep=0.8pt}
}
\begin{tikzpicture}[level distance=1.0cm, sibling distance=2.0cm, every child/.style={line width=0.6pt}]
\node[root] (phi) {$\phi$}
child {node[linear] {}
child {node[linear] {}
child {node[branching] {}
child {node[blpoint] {}
child {node[linear] {}
child {node[branching] {}
child {node[cont] {$\cdots$}}
child {node[cont] {$\cdots$}}
}
}
}
child {node[branching] {}
child {node[cont] {$\cdots$}}
child {node[cont] {$\cdots$}}
}
}
}
};
\end{tikzpicture}
\caption{\small A schematic tree illustrating linear points (white), branching points (dark gray), and branching-linear points (light gray) in a tree with binary offspring law. The structure continues infinitely beyond the dotted nodes.}
\label{fig:tree-structure}
\end{figure}

The following lemma, which is proved in Section~\ref{sec:ProofLemma}, characterises the possible types of vertices in $\Tt_{\infty}$. See Fig.\ \ref{fig:tree-structure} for an illustration.

\begin{lemma}{\bf [Type characterisation lemma]}
\label{lem1}
Consider a Galton--Watson tree with the offspring law $(p_k)_{k\in\N_0}$ satisfying $p_1 =p\in (0,1)$ and $p_2= 1 -p$.
\begin{itemize}
\item[\rm{(a)}] 
The root can only be either positive or neutral. It is neutral if and only if it is a branching point and its two children are branching-linear points.
\item[\rm{(b)}] 
A linear point is negative if and only if it is the child of the root and is not the parent of a branching point. It is positive if and only if it is the parent of a branching point and is not the child of the root.
\item[\rm{(c)}] 
A branching-linear point is neutral if and only if its parent is the root and its child is a linear point, otherwise, it is positive.
\item[\rm{(d)}] 
A branching point that is not the root is either negative or neutral. It is neutral if and only if its parent and both of its children are also non-root branching points.
\end{itemize}
\end{lemma}


\subsection{Exploration along a random path}
\label{sec:expl}

We next describe an exploration of the Galton--Watson tree via a random path along successive linear pieces. For this purpose we generate a sequence of \emph{active} vertices $(\phi_j)_{j\in\N}$, and for each $\phi_j$ denote by
\[
\ell_j = \ell(\phi_j), \qquad \tau_j=\tau(\phi_j),
\]
the length and the terminal vertex of the corresponding linear piece. At $\tau_j$ we choose one of its two children \emph{uniformly at random} and declare it to be the next active vertex $\phi_{j+1}$, while the other child becomes passive $\phi_{j+1}^\prime$. Starting from the root $\phi=\phi_1$ as the first active vertex, we thus construct a growing sequence of random paths
\[
\mathcal{T}_{1} \subseteq \mathcal{T}_{2} \subseteq \,\cdots \, (\subseteq \mathcal{T}_{\infty}).
\]
From each passive vertex $\phi_j^\prime$ we subsequently explore the associated linear piece, of length $\ell_j^\prime=\ell(\phi_j^\prime)$, until we reach its terminal vertex $\tau_j^\prime=\tau(\phi_j^\prime)$. The passive vertices $(\phi_j^\prime)_{j\in\N_2}$ and their descendants are viewed as \emph{ghosts}, in the sense that they do not belong to the active path.

We next describe the exploration algorithm that generates the random path together with the associated ghost vertices and edges.


\paragraph{EXPLORATION ALGORITHM:}

\begin{enumerate}[(1)]
\item 
\textbf{Initialisation.} 
Set $\phi_1=\phi$ (the root). Consider $\Tt_0$ as a path with only one vertex $\phi_1$. Put $j=1$.	
\item \textbf{Step \boldmath{$j$}:}
\begin{enumerate}[(2-1)]  
\item 
\textbf{Active exploration.} 
Starting from $\phi_j$, follow the linear piece until reaching the first branching vertex $\tau_j=\tau(\phi_j)$. Record $\ell_j=\ell(\phi_j)$. At $\tau_j$ let the two children be $\{\tau_{j,1},\tau_{j,2}\}$. Choose one of them uniformly at random (defined on the same probability space $(\Omega,\mathscr{F},\PP)$ and independent of the tree and of the past exploration) and declare it to be the next active vertex $\phi_{j+1}$, while the other becomes the next passive vertex $\phi_{j+1}^\prime$, i.e.,
\[
\phi_{j+1}=
\begin{cases}
\tau_{j,1}, & \text{with probability } \tfrac12\\
\tau_{j,2}, & \text{with probability }\tfrac1 2
\end{cases},
\qquad 
\phi_{j+1}^\prime= \{\tau_{j,1},\tau_{j,2}\}\setminus\{\phi_{j+1}\}.
\]
Add to the explored path $\mathcal{T}_{j-1}$ all vertices and edges of the linear piece starting at $\phi_j$, together with the new active vertex $\phi_{j+1}$ and the edge connecting it to $\tau_j$. The resulting graph is the new path $\mathcal{T}_{j}$.
\item 
\textbf{Passive exploration (if \boldmath{$j\ge2$}).} 
Starting from $\phi_{j}^\prime$, follow the linear piece and connecting edges until reaching $\tau_{j}^\prime=\tau(\phi_j^\prime)$. Record $\ell_j^\prime=\ell(\phi_j^\prime)$. 
All vertices and edges discovered here are \emph{ghosts} and do not enlarge $\mathcal T_j$.
\end{enumerate}
\item 
\textbf{Iteration.} 
Increase $j$ by $1$ and return to (2).
\end{enumerate}


\subsection{Counting process for the vertex types} 
\label{sec:typecount}

We define three random variables that, at each step of the exploration, record the number of vertices of the three types encountered along the path. For $j\in\N$, let
\[
V_j = \big\{u \in V(\Tt_j) : \text{the type of $u$ is determined by step $j$}\big\}.
\]
For $\chi\in S$ we set
\[
N_j^\chi = \big|\big\{u\in V_j: \mathrm{sign}(\Delta_u)=\chi\big\}\big|, 
\qquad 
N_j = |V_j| = N_j^- + N_j^0 + N_j^+,
\]
where $\Delta_u$ denotes the friendship bias of $u$ in $\Tt_\infty$.

At each step $j\in\N$, all vertices of $V(\Tt_j)$ belong to $V_{j}$, except for possibly two vertices: the branching point $\tau_{j}$ and its offspring, $\phi_{j+1}$. The types of $\phi_{j+1}$ remain undetermined at this step, while by Lemma~\ref{lem1}(d), if $\ell_j>0$, then the branching point $\tau_{j}$ also belongs to $V_j$ and its type is negative. Here we note that $\phi_{j}$ is a branching point if and only if $\ell_{j}=0$. More precisely, at each step $j\in\N$, we proceed as follows. First, we examine the vertices $\phi_{j}$ and $\tau_j$ and determine their types whenever possible. Next, we assign the types of the linear points on the path from $\phi_j$ to $\tau_j$, if such points exist. Finally, for $j\in\N_2$, in the case $\ell_{j-1}=0$, where $\tau_{j-1}=\phi_{j-1}$, we determine the type of $\tau_{j-1}$ provided it has not already been specified. 

Note that $\{\ell_j\}_{j\in\N}$ and $\{\ell_j^\prime\}_{j\in\N_2}$ are two independent families of i.i.d.\ random variables. 
Each variable represents the number of consecutive vertices with exactly one child before a branching point is reached. 
Consequently, both $\ell_j$ and $\ell_j^\prime$ follow a geometric law on $\N_0$ with success parameter $1-p$, i.e.,
\[
\PP(\ell_j = k) = \PP(\ell_j^\prime = k) = p^k(1-p), \qquad k \in \N_0.
\]

The types of vertices are determined step by step from the values of $\ell_j$ (and, if necessary, $\ell_j'$) as follows:

\begin{itemize}
	\item \textbf{Step 1:}
	\begin{itemize}
	\item If $\ell_1=0$, then $\phi_{1}$ is itself the branching point $\tau_{1}$. Its type remains undetermined at this step.  
	\item  If $\ell_1>0$, then by Lemma~\ref{lem1}(a) $\phi_{1}$ is positive. 
	\item  If $\ell_1>0$, then by Lemma~\ref{lem1}(d) $\tau_{1}$ is negative.
	\item  If $\ell_1=2$, then by Lemma~\ref{lem1}(b) the child of the root $\phi_1$ is neutral. 
	\item  If $\ell_1>2$, then by Lemma~\ref{lem1}(b) the child of the root $\phi_1$ is negative, the parent of the first branching point $\tau_1$ is positive, and the remaining $\ell_1-3$ linear points on the path from $\phi_{1}$ to $\tau_{1}$ are neutral.
	\end{itemize}
\end{itemize}
Thus, at step $1$ we obtain
\begin{align*}
	&N_{1}^{-}=\mathbbm{1}_{\{\ell_{1}>0\}}+\mathbbm{1}_{\{\ell_{1}>2\}},\\
	&N_{1}^{0}=\ell_{1}-\mathbbm{1}_{\{\ell_{1}>0\}}-2\,\mathbbm{1}_{\{\ell_{1}>2\}},\\
	&N_{1}^{+}=\mathbbm{1}_{\{\ell_{1}>0\}}+\mathbbm{1}_{\{\ell_{1}>2\}}.
\end{align*}

\begin{itemize}
	\item \textbf{Step 2:}
	\begin{itemize}
	\item If $\ell_2=0$, then $\phi_{2}$ is itself the branching point $\tau_{2}$. By Lemma~\ref{lem1}(d), its type is negative when $\ell_{1}=0$, while for $\ell_{1}>0$ its type remains undetermined at this step.
	\item If $\ell_2>0$, then  by Lemma~\ref{lem1}(c), $\phi_{2}$ is neutral when $\{\ell_1=0,\ell_2>1\}$, and positive otherwise.
	\item If $\ell_2>0$, then by Lemma~\ref{lem1}(d) the new branching point $\tau_{2}$ is negative.
	\item If $\ell_2>1$, then by Lemma~\ref{lem1}(b) the parent of $\tau_2$ is positive, while the other $\ell_2-2$ linear points on the path from $\phi_{2}$ to $\tau_{2}$ are neutral.
	\item If $\{\ell_1=\ell_2=0\}$ or $\{\ell_1=0,\ell_2>0,\ell_2^\prime=0\}$, then by Lemma~\ref{lem1}(a) the root $\tau_{1}=\phi_1$ is positive.  
	\item If $\{\ell_1=0,\ell_2>0,\ell_2^\prime>0\}$, then by Lemma~\ref{lem1}(a) the root $\tau_1=\phi_1$ is neutral.  
	\end{itemize}
\end{itemize}
Thus, at step $2$ we obtain
\begin{align*}
	&N_{2}^{-}=N_{1}^{-}+\mathbbm{1}_{\{\ell_{1}=0,\, \ell_{2}=0\}}+\mathbbm{1}_{\{\ell_{2}>0\}},\\
	&N_{2}^{0}=N_{1}^{0}+\mathbbm{1}_{\{\ell_{1}=0,\, \ell_{2}>1\}}+\ell_{2}-\mathbbm{1}_{\{\ell_{2}>0\}}-\mathbbm{1}_{\{\ell_{2}>1\}}+\mathbbm{1}_{\{\ell_{1}=0,\, \ell_{2}>0,\,\ell_2^\prime>0\}},\\
	&N_{2}^{+}=N_{1}^{+}+\mathbbm{1}_{\{\ell_{2}>0\}}-\mathbbm{1}_{\{\ell_{1}=0,\, \ell_{2}>1\}}+\mathbbm{1}_{\{\ell_{2}>1\}}+\mathbbm{1}_{\{\ell_{1}=0,\, \ell_{2}=0\}}+\mathbbm{1}_{\{\ell_{1}=0,\, \ell_{2}>0,\,\ell_2^\prime=0\}}.
\end{align*}

\begin{itemize}
	\item \textbf{Step $\boldsymbol{j\in\mathbb{N}_3}$:}    
	\begin{itemize}
	\item If $\ell_j=0$, then $\phi_{j}$ is itself the branching point $\tau_{j}$. Its type remains undetermined at this step. 
	\item If $\ell_j>0$, then by Lemma~\ref{lem1}(c) $\phi_j$ is positive. 
	\item If $\ell_j>0$, then by Lemma~\ref{lem1}(d) the new branching point $\tau_j$ is negative.  
	\item If $\ell_j>1$, then by Lemma~\ref{lem1}(b) the parent of $\tau_j$ is positive, while the other $\ell_j-2$ linear points on the path from $\phi_{j}$ to $\tau_{j}$ are neutral.
	\item For $j=3$, if $\{\ell_{j-2}>0,\ell_{j-1}=\ell_{j}=\ell_j^\prime=0\}$, then by Lemma~\ref{lem1}(d) the vertex $\tau_{j-1}=\phi_{j-1}$ is neutral.
	\item For $j=3$, if $\{\ell_{j-2}>0,\ell_{j-1}=0,\ell_j>0\}$ or $\{\ell_{j-2}>0,\ell_{j-1}=\ell_{j}=0,\ell_j^\prime>0\}$, then by Lemma~\ref{lem1}(d) the vertex $\tau_{j-1}=\phi_{j-1}$ is negative.
	\item For $j\in\N_4$, if $\{\ell_{j-1}=\ell_{j}=\ell_j^\prime=0\}$, then by Lemma~\ref{lem1}(d) the vertex $\tau_{j-1}=\phi_{j-1}$ is neutral.
	\item For $j\in\N_4$, if $\{\ell_{j-1}=0,\ell_j>0\}$ or $\{\ell_{j-1}=\ell_{j}=0,\ell_j^\prime>0\}$, then by Lemma~\ref{lem1}(d) the vertex $\tau_{j-1}=\phi_{j-1}$ is negative.
	\end{itemize}
\end{itemize}
Thus, at step~$j\in \N_3$ we obtain
\begin{equation*}
	\begin{aligned}
&N_j^- = N_{j-1}^{-} + \mathbbm{1}_{\{\ell_j>0\}}
		+ \begin{cases}
			\mathbbm{1}_{\{\ell_{j-2}>0,\, \ell_{j-1}=0,\, \ell_j>0\}}
			+ \mathbbm{1}_{\{\ell_{j-2}>0,\, \ell_{j-1}=0,\, \ell_j=0,\, \ell_j^\prime>0\}}, & j=3, \\[4pt]
			\mathbbm{1}_{\{\ell_{j-1}=0,\, \ell_j>0\}}
			+ \mathbbm{1}_{\{\ell_{j-1}=0,\,\ell_j=0,\, \ell_j^\prime>0\}}, & j\in\N_{4},
		\end{cases} \\[6pt]
&N_j^0 = N_{j-1}^0 + \ell_j-\mathbbm{1}_{\{\ell_{j}>0\}}-\mathbbm{1}_{\{\ell_{j}>1\}}
		+ \begin{cases}
			\mathbbm{1}_{\{\ell_{j-2}>0,\, \ell_{j-1}=0,\,\ell_j=0,\,\ell_j^\prime=0\}}, & j=3, \\[4pt]
			\mathbbm{1}_{\{\ell_{j-1}=0,\,\ell_j=0,\,\ell_j^\prime=0\}}, & j\in\N_{4},
		\end{cases} \\[6pt]
&N_j^+ = N_{j-1}^{+} + \mathbbm{1}_{\{\ell_j>0\}} + \mathbbm{1}_{\{\ell_j>1\}} .
	\end{aligned}
\end{equation*}

Note that the above expressions define a recursion formula for the triple $(N_j^-,N_j^0,N_j^+)$, which depends only on $\ell_j$ when $j=1$, on $(N_{j-1}^-,N_{j-1}^0,N_{j-1}^+)$ together with $(\ell_{j-1},\ell_j,\ell_j^\prime)$ when $j=2$, and on $(N_{j-1}^-,N_{j-1}^0,N_{j-1}^+)$ together with $(\ell_{j-2},\ell_{j-1},\ell_j,\ell_j^\prime)$ when $j\in\N_3$. Define
\[
\Pi_0 = (0,0,0), \qquad \Pi_j = (N_j^-,N_j^0,N_j^+), \quad j\in\N.
\]
By construction, there exists a deterministic function 
\[
g\colon\,\N_0^3\times\N_0^4\to\N_0^3
\]
such that
\[
\Pi_j = g(\Pi_{j-1},Y_j),
\]
where
\[
Y_1 = (0,0,\ell_1,0), \qquad 
Y_2 = (0,\ell_1,\ell_2,\ell_2^\prime), \qquad 
Y_j = (\ell_{j-2},\ell_{j-1},\ell_j,\ell_j^\prime), \quad j\in\N_3.
\]
Since $\{Y_j\}_{j\in\N}$ is a time-homogeneous Markov chain, the joint process $\{(\Pi_{j},Y_j)\}_{j\in\N}$ is also a time-homogeneous Markov chain.


\subsection{LDP for the vertex types along a random path up to a given branching depth}
\label{sec:LDP}


\paragraph{LDP.}

Our first main theorem, which is proved in Section~\ref{sec:ProofLDP1}, establishes an LDP for the sequence $\Upsilon = \{\Upsilon_j\}_{j\in\N_4}$ given by 
\[
\Upsilon_j = \left(\frac{N_j}{j}, \frac{\Pi_j}{N_j}\right) \in (0,\infty) \times \mathcal{S}_2.
\]
To state and derive the theorem we need an LDP for the empirical law of the extended Markov process $W=\{W_j\}_{j\in\N_4}$ given by
\[
W_j = \big((\ell_{j-2},\ell_{j-2}^\prime),(\ell_{j-1},\ell_{j-1}^\prime), (\ell_j,\ell_j^\prime)\big)\in E,
\]
where $E=(\N_0^2)^3$. We write $\zeta$ to denote the geometric law on $\N_0$ with success parameter $1-p$. 
 
\begin{theorem}{\bf [LDP for type fractions along a random path]}
\label{LDP1}
Under the joint law of the Galton--Watson tree $\mathcal{T}_\infty$ and the random exploration path, the family of random variables $\{\Upsilon_j\}_{j\in\N_4}$ satisfies the LDP in the Polish space $(0,\infty) \times \mathcal{S}_2$ with rate $j$ and with good rate function $J$ given by
\[
J(x,s) = \inf_{\nu \in \Phi(x,s)} I(\nu), \qquad x \in (0,\infty),\, s \in \mathcal{S}_2.
\]
Here,
\begin{itemize}
\item 
$I\colon\,\mathcal{P}(E) \to [0,\infty]$ is the good rate function given by
\[
I(\nu) =
\begin{cases}
H\left(\nu\,\middle\|\, \pi_{1,2}\nu \otimes \zeta^{\otimes 2}\right), & \text{if } \pi_{1,2}\nu=\pi_{2,3}\nu,\\[1ex]
\infty, & \text{otherwise,}
\end{cases}
\]
where $H(\cdot\, \|\, \cdot)$ denotes relative entropy and $\pi_{1,2}$, $\pi_{2,3}$ denote the projection onto the first two, respectively, the last two of the three coordinates.
\item
The constraint on $\nu$ is to the subset of $\mathcal{P}(E)$ given by
\[
\Phi(x,s) = \big\{\nu \in \mathcal{P}(E)\colon\,\|G(\nu)\|_1 = x,\,G(\nu)/\|G(\nu)\|_1 = s\big\},
\qquad x \in (0,\infty),\,s \in \mathcal{S}_2,
\]
where $G = (G^\chi)_{\chi \in S}$ with $G^\chi\colon\,\mathcal{P}(E) \to [0,\infty]$ is given by
\[
G^\chi(\nu) = \int_{E} h^\chi\,\mathrm{d}\nu
\]
with $h^\chi\colon\,E \to \N_0$ given by
\[
\begin{aligned}
h^{-}\big((a,a^\prime),(b,b^\prime),(c,c^\prime)\big)
&= \mathbbm{1}_{\{c>0\}}
+ \mathbbm{1}_{\{b=0,\, c>0\}}
+ \mathbbm{1}_{\{b=0,\,c=0,\, c^\prime>0\}}, \\[6pt]
h^{0}\big((a,a^\prime),(b,b^\prime),(c,c^\prime)\big)
&= c-\mathbbm{1}_{\{c>0\}}-\mathbbm{1}_{\{c>1\}}
+ \mathbbm{1}_{\{b=0,\,c=0,\,c^\prime=0\}}, \\[6pt]
h^{+}\big((a,a^\prime),(b,b^\prime),(c,c^\prime)\big)
&= \mathbbm{1}_{\{c>0\}} + \mathbbm{1}_{\{c>1\}}.
\end{aligned}
\]
\end{itemize}
\end{theorem}

Our second main theorem, which is proved in Section~\ref{sec:ProofLDP2}, establishes key properties of the rate function. Let
\[
\cD = \{(x,s) \in (0,\infty) \times \mathcal{S}_2\colon J(x,s)<\infty\}
\]
be the effective domain of $J$. 

\begin{theorem}{\bf [Properties of the rate function]}
\label{LDP2}
\begin{itemize}
\item[{\rm (a)}] On $\cD$, $(x,s) \mapsto J(x,s)$ is marginally convex and achieves its unique zero at $(x^*,s^*)$ given by
\[
(x^*,s^*) = \big(\|G(\nu^*)\|_1,G(\nu^*)/\|G(\nu^*)\|_1\big)
\] 
with $\nu^* = (\zeta^{\otimes 2})^{\otimes 3}$, which equals
\[
x^* = \frac{1}{1-p}, \qquad s^* = \big((1-p)-(1-p)^4,\, p^3 + (1-p)^4,\, p-p^3\big).          
\]
\item[{\rm (b)}] 
For every $(x,s) \in \mathrm{int}(\cD)$, the infimum defining $J(x,s)$ is achieved and equals
\[
J(x,s) = -\lambda(x,s) - x \sum_{\chi \in S} s^\chi \lambda^\chi (x,s), 
\]
where $\lambda(x,s)$ and $\lambda^\chi(x,s)$, $\chi \in S$, are Lagrangian multipliers that uniquely solve four equations involving rational functions whose coefficients depend on $(x,s)$ and $p$ (see Section \ref{sec:ProofLDP2}).
\item[{\rm (c)}]
On $\mathrm{int}(\cD)$, $(x,s) \mapsto J(x,s)$ is analytic and marginally strictly convex.
\item[{\rm(d)}] 
$\mathrm{int}(\cD) = \big\{(x,s)\in(0,\infty)\times\mathcal S_2\colon\, 1 < x < \infty,\,s^- < 1/x < s^-+s^0 < 1 < 3s^-+s^0\big\}$.
\end{itemize}
\end{theorem}

The following law of large numbers (LLN) is immediate from the fact that $J$ has a unique zero. 

\begin{corollary}{\bf [LLN for type fractions along a random path]}
\label{LLN}
The sequence $\Upsilon$ converges to $(x^*,s^*)$ almost surely.
\end{corollary}


\paragraph{Discussion.}

Theorem \ref{LDP1} shows that the large deviation properties of the sequence $\Upsilon$ can be analysed through the large deviation properties of the empirical law of the Markov process $W$. The rate function $J$ of the former is an infimum of the rate function $I$ of the latter subject to certain constraints that are given by the functional $G$. This functional is an integral with respect to the measure $\nu$ of the function $h$ dictated by the exploration algorithm, and records how the path of the Markov chain $W$ gives rise to the deviation $(x,s)$ of interest. 

Theorem \ref{LDP2} identifies the key properties of the rate function $J$. 
\begin{itemize}
\item
The fact that, according to Part (a), $J$ has a unique zero guarantees that the strong law of large numbers holds, as stated in Corollary \ref{LLN}. The LDP does not imply the central limit theorem. However, if the latter would hold, then the variance would scale like $j$ times the inverse curvature of the rate function at its unique zero, according to the folklore of large deviation theory \cite[Chapter I]{FdH}. The curvature can be read off from the Hessian matrix of the Lagrange multipliers (which are identified in Section~\ref{sec:ProofLDP2}). 
\item
Part (b) allows us to \emph{numerically compute} the rate function (as explained in Remark~\ref{rem:numcomp} and Figs.~\ref{fig:plot1}--\ref{fig:plot2}). 
\item
The marginal strict convexity in Part (c) is an important property, because it allows us to extend the LDP in Theorem \ref{LDP1} to the setting where, instead of a \emph{single} downward random path, we consider a subtree consisting of a \emph{finite} number of downward random paths. Suppose that for fixed $\ell$ we run a downward random path of branching depth $\alpha_0\ell$, subsequently split this path and run two downward random paths of branching depths $\alpha_{10}\ell$ and $\alpha_{11}\ell$, and so on. Suppose that we fix $\alpha_0,\alpha_{10},\alpha_{11}$, and so on, to be all strictly positive, and let $\ell \to \infty$. Then the sequences $\Upsilon$ associated with the different pieces of the tree are asymptotically independent as $\ell\to\infty$. Hence their joint large deviations factorise and, by the marginal strict convexity of the rate function, a joint large deviation is optimal when all the pieces make the \emph{same} large deviation. Thus, by counting the vertex types and dividing by the total size of the subtree, we get the \emph{exact same} LDP.
\item
Part (d) of Theorem \ref{LDP2} shows that $\mathrm{int}(\cD)$ does not depend on $p$ and is therefore universal. It makes no statement about the relation between $\cD$ and the closure of $\mathrm{int}(\cD)$. The computations and the numerics in Section~\ref{sec:ProofLDP2} suggest that $\cD$ is closed, i.e., $J(x,s)$ converges to a finite limit as $(x,s) \to \partial\cD$ within $\mathrm{int}(\cD)$. 
\end{itemize}
 

\subsection{Comparison with exploration according to generation} 
\label{sec:comparison}

In \cite{HHLP} we computed the typical fractions of vertex types $(f^\chi)_{\chi \in S}$ according to \emph{generation depth}, i.e., all the vertices up to a given generation. For a general offspring law $(p_k)_{k\in\N_0}$ under the restriction that $p_0=0$, $p_1<1$ and $\sum_{k\in\N} (k\log k) p_k<\infty$, we found that
\[
f^\chi = \sum_{k \in \N_0} p_k\,\PP\big({\rm sign}\big[\tX+S_k-k(k+1)\big] = \chi\big),
\]
where $S_k= \sum_{l=1}^k X_l$ with $(X_l)_{l=1}^k$ drawn independently from the offspring law, and $\tilde{X}$ drawn independently from the size-biased offspring law. For the special case of \emph{binary branching}, we have
\[
\PP(X_1 = k)  = p\,\mathbbm{1}_{\{k=1\}} + (1-p)\,\mathbbm{1}_{\{k=2\}},  \quad 
\PP(\tilde{X} = k) = \frac{p\,\mathbbm{1}_{\{k=1\}} + 2(1-p)\,\mathbbm{1}_{\{k=2\}}}{p+2(1-p)}, \quad k \in \N_0.
\]
Substitution gives
\[
f^\chi = p\,\PP\big({\rm sign}\big[\tX+X_1-2\big] = \chi\big) + (1-p)\,\PP\big({\rm sign}\big[\tX+X_1+X_2-6\big] = \chi\big).
\]
Clearly,
\[
\PP\big({\rm sign}\big[\tX+X_1-2\big] = \chi\big) = \left\{\begin{array}{ll}
0, &\chi = -,\\
\tfrac{p^2}{2-p},  &\chi = 0,\\
1-\tfrac{p^2}{2-p}, &\chi = +,
\end{array}
\right.
\]
and
\[ 
\PP\big({\rm sign}\big[\tX+X_1+2-6\big] = \chi\big) = \left\{\begin{array}{ll}
1-\tfrac{2(1-p)^2}{2-p}, &\chi = -,\\
\tfrac{2(1-p)^2}{2-p},  &\chi = 0,\\
0, &\chi = +.
\end{array}
\right.
\]
Hence
\[
f^- = (1-p)\big[1-\tfrac{2(1-p)^2}{2-p}\big], \quad f^0 = 1- f^- - f^+, \quad f^+ = p\big[1-\tfrac{p^2}{2-p}\big].
\]
Note that these do \emph{not} coincide with the typical fractions that we found for our exploration along a random path up to a given  \emph{branching depth}. 

\begin{remark}{\bf [Other forms of exploration]}
{\rm We expect that the LDP depends sensitively on the type of exploration. For instance, we may consider an exploration that includes \emph{all} the vertices up to a given branching depth, not just those along a random path, and ask what is the associated LDP for their types. We expect that an LDP holds, but with a different rate function. We do not pursue such generalisations in the present paper because they are harder. The reason is that the type of a vertex depends on the degrees of all its neighbours and their neighbours, which on a tree can lie in future generations. On trees, such future generations are comparable in size to the union of all the previous generations.}
\end{remark} 


\subsection{Beyond binary branching}
\label{sec:beyondbinary}

The exploration in Section~\ref{sec:exploration} can be extended beyond binary branching. Let $(p_k)_{k\in\N_0}$ be an offspring law satisfying
\[
p_0=0, \qquad p_1\in (0,1), \qquad p_k=0 \quad \forall\, k>m ,
\]
for some $m \in \N_3$. Under this assumption the Galton--Watson tree is infinite and locally finite almost surely. In particular, each vertex has either one child (a linear step) or $k \in \{2,\dots,m\}$ children (a branching step).

Linear pieces are defined as before: starting from a vertex $v$, we follow successive vertices having exactly one child until the next vertex $\tau$ with $X_\tau \geq 2$ children, which is a branching point. At this branching point, one child is chosen uniformly at random on the underlying probability space $(\Omega,\mathscr{F},\PP)$, independently of the tree and of the past explorations, to continue the active exploration path, while the remaining $k-1$ children are declared passive (ghost) roots. As in the binary case, passive subtrees are explored only insofar as needed to determine the vertex types along the active path.

Compared to the binary case, a single branching step may now create more than one passive subtree. Consequently, the state space of the block process must be enlarged to encode, at each branching step, the number of offspring together with the lengths of the linear pieces generated at that step. Concretely, let $\tau_j$ be the $j$th branching point encountered by the active exploration, and let $X_{\tau_j} = K_j \in \{2,\dots,m\}$. After we select the active child, we index this child as $\phi_{j,0}$, and index the remaining $K_j-1$ offspring as $\phi_{j,1},\dots,\phi_{j,K_j-1}$ in increasing order of their Ulam--Harris labels. Define the step variable
\[
\xi_j = \Big(K_j, \ell_{j}^{(0)}, \ell_{j}^{(1)},\dots,\ell_{j}^{(K_j-1)}\Big)
\in \bigcup_{k=2}^{m} \{k\}\times\N_0^{k}, \qquad j \in \N,
\]
where $\ell_{j}^{(r)}$ is the length of the linear piece starting from $\phi_{j,r}$ (with $r=0$ for the active child and $r \geq 1$ for the passive children). The sequence $\{\xi_j\}_{j\in \N}$ is i.i.d.\ under $\PP$, because the offspring numbers and the descendant subtrees of different children are independent in a Galton--Watson tree. 

To determine vertex types, we need local information spanning a bounded number of successive steps, exactly as in the binary case. Therefore we can build a time-homogeneous Markov chain by working with overlapping blocks of successive $\xi_j$'s (analogous to the $3$-block process $W$ in Section~\ref{sec:LDP}). The counting process of vertex types can be written as a measurable function of this Markov chain. The analysis extends along the same lines as in the binary case and leads to a rate function with a variational structure similar to that of Theorem~\ref{LDP1}, namely, a relative entropy minimised under linear constraints determined by the exploration algorithm (which determines new functions $h^\chi$, $\chi \in S$). Since the binary case already leads to a somewhat complex rate function, we do not pursue the general offspring setting further.


\section{Proofs}
\label{sec:proofs}

Section~\ref{sec:ProofLemma} contains the proof of Lemma~\ref{lem1}, Section~\ref{sec:ProofLDP1} of Theorem~\ref{LDP1}, Section~\ref{sec:ProofLDP2} of Theorem~\ref{LDP2}. The proof of Theorem~\ref{LDP1} requires several approximation steps.


\subsection{Proof of the type characterisation lemma}
\label{sec:ProofLemma}

In this section we provide the proof of Lemma~\ref{lem1}. 

\begin{proof}
We prove parts (a)--(d) of the lemma separately.

\medskip\noindent 
(a) Consider the root $\phi$. Each child $\phi_j$ of $\phi$ has degree $d_{\phi_j} = 1+X_{\phi_j}$ (one edge to the root and $X_{\phi_j}$ edges to its own children). Hence
\[
\Delta_\phi
=\frac{1}{d_{\phi}}\sum_{j=1}^{d_{\phi}}(1+X_{\phi_j})-d_{\phi}
=\frac{1}{d_{\phi}}\sum_{j=1}^{d_{\phi}}X_{\phi_j}-(d_{\phi}-1).
\]
If $d_{\phi} = 1$, then $\Delta_\phi = X_{\phi_1}\in\{1,2\}$, and $\Delta_\phi>0$. If $d_{\phi}=2$, then $\Delta_\phi = \tfrac12(X_{\phi_1}+X_{\phi_2})-1$. Since $X_{\phi_1},X_{\phi_2} \in \{1,2\}$, the sum $X_{\phi_1}+X_{\phi_2} \in \{2,3,4\}$, so $\Delta_\phi\in\{0,\tfrac12,1\}$. In particular, $\Delta_\phi \ge 0$ always, and $\Delta_\phi=0$ if and only if $X_{\phi_1}=X_{\phi_2}=1$, which means that $\phi_1$ and $\phi_2$ are branching-linear points. Therefore, for $d_{\phi}=2$ the root is neutral if and only if both its children are branching-linear points.

Combining the two cases $d_{\phi}=1$ and $d_{\phi}=2$, the root can only be positive or neutral, and it is neutral precisely in case it is a branching point and its two children are branching-linear points. This proves Lemma~\ref{lem1}(a).
		
\medskip\noindent 
(b) Let $u\neq \phi$ be a linear point, i.e., $d_u=2$ and the parent of $u$ is not a branching point. Write $u_{-1}$ for the parent of $u$, and consider $u_1$ as the unique child of $u$.
Then
\[
\Delta_u = \tfrac{1}{2}\big(d_{u_{-1}}+1+X_{u_1}\big)-2.
\]
If $u_{-1}=\phi$, then $d_{u_{-1}}=1$, and hence
\[
\Delta_u =\tfrac{1}{2}X_{u_1}-1.
\]
Thus, $\Delta_u<0$ if $X_{u_1}=1$, i.e., if $u$ is not the parent of a branching point. Moreover, $\Delta_u=0$ if $X_{u_1}=2$, i.e., if $u$ is the parent of a branching point. If $u_{-1}\neq\phi$, then $d_{u_{-1}}=2$, and hence
\[
\Delta_u=\tfrac{1}{2}X_{u_1}-\tfrac{1}{2}.
\]
Thus, $\Delta_u>0$ if $X_{u_1}=2$, i.e., if $u$ is the parent of a branching point, while $\Delta_u=0$ if $X_{u_1}=1$. Combining the two cases, we conclude that a linear point is negative if and only if it is the child of the root and not the parent of a branching point, and it is positive if and only if it is the parent of a branching point and not the child of the root. In other cases it is neutral. This proves Lemma~\ref{lem1}(b).

\medskip\noindent 
(c) Let $u$ be a branching-linear point, i.e., $d_u=2$ with a branching parent $u_{-1}$ and a unique child $u_{1}$. Then
\[
\Delta_u = \tfrac{1}{2}\big(d_{u_{-1}} + d_{u_{1}}\big) - 2.
\]
First, let $u_{-1}=\phi$, i.e., $d_{u_{-1}}=2$. If $u_{1}$ is a linear point, then $d_{u_{1}}=2$, so $\Delta_u=0$, and $u$ is neutral. If instead $u_{1}$ is a branching point, then $d_{u_{1}}=3$, so $\Delta_u=\frac{1}{2}>0$, and $u$ is positive. Next, let $u_{-1}\neq \phi$, i.e., $d_{u_{-1}}=3$. If $u_{1}$ is a linear point, then $d_{u_{1}}=2$, so $\Delta_u=\frac{1}{2}>0$. If $u_{1}$ is a branching point, then $d_{u_{1}}=3$, hence $\Delta_u=1>0$. Therefore, in both cases $u$ is positive. Hence a branching-linear point is neutral if and only if its parent is the root and its child is a linear point; otherwise it is positive. This proves Lemma~\ref{lem1}(c).

\medskip\noindent 
(d) Let $u\neq \phi$ be a branching point, so $d_u=3$ with the parent $u_{-1}$ and two children $u_1$ and $u_2$. The parent $u_{-1}$ can have degree $d_{u_{-1}}\in\{1,2,3\}$, while each child $u_i$ has degree $d_{u_i}\in\{2,3\}$. Therefore
\[
\Delta_u = \tfrac{1}{3}\big(d_{u_{-1}}+d_{u_1}+d_{u_2}) - 3 \leq \tfrac{1}{3}(3+3+3)-3=0.
\]
Thus $\Delta_u\leq 0$, so $u$ is either negative or neutral. Moreover, $\Delta_u=0$ if and only if $d_{u_{-1}}=d_{u_1}=d_{u_2}=3$, i.e., the parent and both children are also non-root branching points. Hence a branching point that is not the root is either negative or neutral, and it is neutral precisely when its parent and both of its children are also non-root branching points. This proves Lemma~\ref{lem1}(d).
\end{proof}


\subsection{Proof of the first LDP theorem} 
\label{sec:ProofLDP1}

In this section we prove Theorem~\ref{LDP1}. 

\begin{proof}
The proof is divided into 8 Steps and runs via the G\"artner--Ellis theorem \cite[Theorem V.6]{FdH}. In Steps 1--3 we introduce the key underlying Markov chain and state the standard LDP for its empirical measure. In Steps 4--5 we change variables, compute an associated moment generating function in terms of iterates of $2 \times 2$ matrices (Lemma~\ref{lem:matrix-rep} below), and derive the limit of the associated cumulant generating function (Lemma~\ref{lem:existence} below). In Steps 6 we derive the key properties of the limit that are needed for the application of the G\"artner--Ellis theorem (Lemma~\ref{lem:regularity} below). In Step 7 we apply the G\"artner--Ellis theorem to derive the LDP for the sequence of empirical averages of the vertex types (Lemma~\ref{lem:LDP-type-fractions} below). The rate function in this LDP is given by the Fenchel--Legendre transform of the cumulant generating function. In Step 8, finally, we complete the proof of Theorem~\ref{LDP1}.

\paragraph{Step 1.\ Block process and Markov property.}

Consider the independent families $\{\ell_j\}_{j\in\N_{-1}}$ and $\{\ell_j^\prime\}_{j\in\N_{-1}}$ of i.i.d.\ random variables with law $\zeta$ (some of which arise from the exploration process, while others are auxiliary) given by
\[
\zeta(k)=p^k\, (1-p), \qquad k\in\N_0.
\]
For $j\in\N$, define the $3$-block process
\[
W_{j}=((\ell_{j-2},\ell_{j-2}^{\prime}),(\ell_{j-1},\ell_{j-1}^{\prime}),(\ell_{j},\ell_{j}^{\prime})),
\]
where $\{\ell_j\}_{j\in\N_{-1}}$ and $\{\ell_j^{\prime}\}_{j\in\N_{-1}}$ are i.i.d.\ geometric random variables on $\N_0$ and mutually independent. The stochastic process $\{W_j\}_{j\in \N}$ is a time-homogeneous Markov chain on $E$ with transition kernel
\[
P\left(((a,a^\prime),(b,b^\prime),(c,c^\prime)),
((d,d^\prime),(e,e^\prime),(f,f^\prime))\right) 
= \mathbbm{1}_{\{(b,b^\prime)=(d,d^\prime)\}}\,\mathbbm{1}_{\{(c,c^\prime)=(e,e^\prime)\}}\, \zeta(f)\, \zeta(f^\prime)
\]
for $a,b,c,d,e,f,a^\prime,b^\prime,c^\prime,d^\prime,e^\prime,f^\prime\in\N_0$.


\paragraph{Step 2.\ Increments of the counting process.}

To study the evolution of the counting process $\{\Pi_{j}\}_{j\in\N}$, we introduce the measurable functions $h^\chi\colon\,E\to\N_0$ by
\[
\begin{aligned}
&h^{-}\big(((a,a^\prime),(b,b^\prime),(c,c^\prime))\big)
= \mathbbm{1}_{\{c>0\}}
+ \mathbbm{1}_{\{b=0,\, c>0\}}
+ \mathbbm{1}_{\{b=0,\,c=0,\, c^\prime>0\}}, \\[6pt]
&h^{0}\big(((a,a^\prime),(b,b^\prime),(c,c^\prime))\big)
= c - \mathbbm{1}_{\{c>0\}}-\mathbbm{1}_{\{c>1\}}
+ \mathbbm{1}_{\{b=0,\,c=0,\,c^\prime=0\}}, \\[6pt]
&h^{+}\big(((a,a^\prime),(b,b^\prime),(c,c^\prime))\big)
= \mathbbm{1}_{\{c>0\}} + \mathbbm{1}_{\{c>1\}} .
\end{aligned}
\]
For $\chi\in S$, the increments $N_j^{\chi} - N_{j-1}^{\chi} = h^{\chi}(W_j)$ for $j\in\N_{4}$, and
\begin{align*}
 N_3^{\chi} = h^{\chi}(W_1) + h^{\chi}(W_2)+h^{\chi}(W_3) + R^{\chi},
\end{align*}
where $R^{\chi}$ is a finite linear combination of indicator functions and hence $\sup_{\omega\in\Omega}|R^{\chi}(\omega)|<\infty$. Thus
\[
N_j^{\chi} = \sum_{i=1}^{j} h^\chi(W_i) + R^{\chi}, \qquad \chi\in S,\quad j\in\N_4.
\]


\paragraph{Step 3.\ LDP for the empirical measure of the block process.}

We next establish an LDP for the empirical measures
\[
L_j = \frac{1}{j} \sum_{i=1}^{j} \delta_{W_{i}}\in\mathcal{P}(E),\qquad j\in\N.
\]
In order to describe the possible limit points of the sequence $\{L_j\}_{j\in\N}$, it is necessary to take into account the overlap structure of the triples. Successive triples $W_i$ and $W_{i+1}$ share the coordinates $(\ell_{i-1},\ell_{i-1}^{\prime})$ and $(\ell_{i},\ell_{i}^{\prime})$, and therefore any limiting law must satisfy a consistency relation between the $(1,2)$-marginal and the $(2,3)$-marginal. More precisely, for $j\in\N$, define
\[
\pi_{1,2}L_j
= \frac{1}{j}\sum_{i=1}^j
\delta_{((\ell_{i-2},\ell_{i-2}^{\prime}),(\ell_{i-1},\ell_{i-1}^{\prime}))},
\qquad
\pi_{2,3}L_j = \frac{1}{j}\sum_{i=1}^j
\delta_{((\ell_{i-1},\ell_{i-1}^{\prime}),(\ell_i,\ell_i^{\prime}))}.
\]
By shifting the index in the first sum, we see that all terms coincide except for the two boundary pairs $((\ell_{-1},\ell_{-1}^{\prime}),(\ell_{0},\ell_{0}^{\prime}))$ and $((\ell_{j-1},\ell_{j-1}^{\prime}),(\ell_j,\ell_j^{\prime}))$. Hence, with $\|\cdot\|_{\mathrm{TV}}$ the total variation norm on $\mathcal{P}(E)$,
\[
\big\|\pi_{1,2}L_j-\pi_{2,3}L_j\big\|_{\mathrm{TV}}
=\frac{1}{2}\sum_{x\in(\N_0^2)^2}\big|\pi_{1,2}L_j(x)-\pi_{2,3}L_j(x)\big|
\leq \frac{1}{2}\times\frac{2}{j} = \frac{1}{j}\to 0,
\qquad j\to\infty.
\]
Since the projection maps
\[ 
\pi_{i,j}\colon\,\mathcal{P}(E)\to\mathcal{P}((\N_0^2)^2),\qquad i\neq j,\quad i,j\in\{1,2,3\}, 
\]
are continuous with respect to the weak topology, this property is preserved under weak limits. Consequently, any weak limit point $\nu$ of $(L_j)_{j\in\N}$ must satisfy the consistency relation
\[
\pi_{1,2}\nu = \pi_{2,3}\nu.
\]
By \cite[Theorem II.18]{FdH} and the consistency relation, the sequence $\{L_j\}_{j\in\N}$, taking values in the Polish space $\mathcal{P}(E)$, satisfies the LDP with rate $j$ and good rate function
\[
I(\nu) =
\begin{cases}
\begin{aligned}[t]
\sum_{x,x^{\prime},y,y^{\prime},z,z^{\prime}\in\N_0}
&\nu\big((x,x^{\prime}),(y,y^{\prime}),(z,z^{\prime})\big)\\[-3mm]
&\times
\log\bigg(\dfrac{\nu\big((x,x^{\prime}),(y,y^{\prime}),(z,z^{\prime})\big)}
{\pi_{1,2}(\nu)\big((x,x^{\prime}),(y,y^{\prime})\big)\,\zeta(z)\,\zeta(z^{\prime})}\bigg),
\end{aligned}
&\text{if } \pi_{1,2}\nu=\pi_{2,3}\nu,\\[15mm]
\infty, 
& \text{otherwise,}
\end{cases}
\]
where
\[
\pi_{1,2}(\nu) \big((x,x^{\prime}),(y,y^{\prime})\big) = \sum_{(z,z^{\prime})\in\N_0^2}
\nu \big((x,x^{\prime}),(y,y^{\prime}),(z,z^{\prime})\big).
\]
Note that if $\nu$ is admissible, in the sense that it satisfies the consistency condition $\pi_{1,2}\nu=\pi_{2,3}\nu$, then $I(\nu) = H(\nu ~\|~ \pi_{1,2}\nu \otimes \zeta^{\otimes 2})$, the relative entropy of $\nu$ with respect to $\pi_{1,2}\nu \otimes \zeta^{\otimes 2}$.


\paragraph{Step 4.\ Functional triple and change of variables.}

For $\nu\in\mathcal{P}(E)$, define the functional triple
\[
G(\nu) = (G^{\chi}(\nu))_{\chi\in S},
\qquad
G^{\chi}(\nu) = \int_{E} h^{\chi} \, \mathrm{d}\nu.
\]
Note that
\[
G\colon\, \mathcal{P}(E) \to F,
\]
with
\[
F = [0,2]\times[0,\infty]\times[0,2].
\]
To metrise $F$, we use the map $\psi\colon\, [0,\infty] \to [0,1]$ given by
\[
\psi(a) = \frac{a}{1+a}, \quad a\in[0,\infty), \qquad \psi(\infty) = 1.
\]
We endow $[0,\infty]$ with the metric
\[
d_{\infty}(a,b) = |\psi(a)-\psi(b)|.
\]
Indeed, $([0,1],\psi)$ is homeomorphic to the one-point (Alexandroff) compactification of $[0,\infty)$. The map $\psi$ is a homeomorphism between $[0,\infty)$ and $[0,1)$, and it extends continuously to a homeomorphism between $[0,\infty]$ and $[0,1]$. Hence, the topology induced by $d_{\infty}$ coincides with the usual topology on $[0,\infty)$, and extends it naturally to include the point $\infty$. We equip $F$ with the product metric
\[
d_F\big((x^-,x^0,x^+),(y^-,y^0,y^+)\big) = |x^{-} - y^{-}| + d_{\infty}(x^{0},y^{0}) + |x^{+} - y^{+}|.
\]
This metric induces the natural product topology on $F$, and since each coordinate space is Polish, the metric space $(F,d_F)$ is Polish as well.
 
From the representation of $N_j^{\chi}$ and the definition of $L_j$, we obtain the vector of normalised counts
\[
Q_{j}=\Big(\frac{N_j^{\chi}}{j}\Big)_{\chi\in S} = G(L_j) + \frac{1}{j}\,R, \qquad j \in \N_4,
\]
where $R = (R^{\chi})_{\chi\in S}$ is a uniformly bounded remainder term, i.e.
\[
\sup_{\omega\in\Omega}\|R(\omega)\|_{1} 
= \sup_{\omega\in\Omega}\left(|R^{-}(\omega)| + |R^{0}(\omega)| + |R^{+}(\omega)|\right) < \infty.
\]
Introduce the map $\Psi\colon\, (0,\infty)^3 \to (0,\infty) \times \mathcal{S}_{2}$ by setting
\[
\Psi(\theta) = \big(\|\theta\|_1,\, \theta/ \|\theta\|_1\big).
\]
Note that $\Psi$ is a homeomorphism with inverse $\Psi^{-1}(x,s)=xs$, and that
\[
\Upsilon_{j}=\Psi(Q_{j}), \qquad j\in\N_4.
\]


\paragraph{Step 5.\ The cumulant generating function.}

Define the empirical type averages
\[
Z_j = \frac{1}{j} \sum_{i=1}^{j} h(W_i),  \qquad j \in \N,
\]
where $h = (h^-,h^0,h^+)$ and $Z_{j} \in F_0$ with
\[
F_0 = [0,2] \times [0,\infty) \times [0,2].
\]
Note that $Z_j = G(L_j)$. For $t = (t_{-},t_{0},t_{+}) \in \R^3$, define the cumulant generating function as the limit
\[
\Lambda(t) = \lim_{j\to\infty} \frac{1}{j} \log \EE\big[\ee^{j \langle t, Z_j \rangle}\big],
\]
where $\langle \cdot,\cdot \rangle$ denotes the standard inner product. Since
\[
j \langle t,Z_j \rangle = \sum_{i=1}^j \langle t,h(W_i) \rangle,
\]
the analysis of $\Lambda$ reduces to the study of an additive functional of the Markov chain $\{W_i\}_{i \in \N}$.

From the definitions of $h^{\chi}$, $\chi \in S$, we see that the vector $h(W_i)$ depends only on whether or not $\ell_{i-1} = 0$ and on the pair $(\ell_i,\ell_i^{\prime})$. Put
\[
\eta_i = \mathbbm{1}_{\{\ell_i>0\}}.
\]
Since $h(W_i)$ depends only on $(\ell_{i-1},\ell_i,\ell_i^{\prime})$ and $\{(\ell_i,\ell_i^{\prime})\}_{i\in\N_{-1}}$ is i.i.d., the conditional law of $(h(W_i),\eta_i)$ given $\eta_{i-1}$ does not depend on $i$. For $t \in \R^3$, define the $2 \times 2$ matrix $M(t) = (M(t)_{r,s})_{r,s \in \{0,1\}}$ by
\[
M(t)_{r,s} = \EE\big[\ee^{\langle t,h(W_1)\rangle}\, \mathbbm{1}_{\{\eta_1 = s\}} \,\big|\, \eta_{0} = r \big].
\]
A straightforward computation gives, for $t_0 < -\log p$,
\[
M(t) =
\begin{pmatrix}
q\,(p\,\ee^{t_-} + q\,\ee^{t_0})
&
q\,p\, \ee^{2t_-+t_+} + \frac{q\,p^2\, \ee^{2t_-+2t_+}}{1 - p\ee^{t_0}}
\\[1.2em]
q
&
q\,p\, \ee^{t_-+t_+} + \frac{q\, p^2\, \ee^{t_- + 2t_+}}{1 - p\ee^{t_0}}
\end{pmatrix},
\]
where $q = 1-p$.

\begin{lemma}{\bf[Matrix representation of the moment generating function]}
\label{lem:matrix-rep}
For every $j\in\N$ and every $t=(t_{-},t_{0},t_{+}) \in \R^3$ with $t_0 < \log(1/p)$,
\[
\EE\big[\ee^{j\langle t , Z_j\rangle}\big] = \mu M(t)^j \mathbf{1},
\]
where
\[
\mu = (q,p), \qquad \mathbf{1}=(1,1)^{\mathsf T}.
\]
\end{lemma}

\begin{proof}
Let $t = (t_{-},t_{0},t_{+}) \in \R^3$ with $t_0 < \log(1/p)$. For $j_1 , j_2 \in\N$ such that $j_1 \leq j_2$, set
\[
A^{j_1 ,j_2}_r = \EE\bigg[\ee^{\sum_{i=j_1}^{j_2} \langle t,h(W_i)\rangle} \,\Big|\, 
\eta_{j_1-1}=r \bigg],\qquad r\in\{0,1\}.
\]
Then
\[ 
\begin{aligned}
A^{j_1 ,j_2}_{\eta_{j_1-1}} 
&= \EE\bigg[\ee^{\sum_{i=j_1}^{j_2} \langle t,h(W_i)\rangle} \,\Big|\, \eta_{j_1-1} \bigg]\\
&= A^{j_1 ,j_2}_0\,\mathbbm{1}_{\{\eta_{j_1-1}=0\}} + A^{j_1 ,j_2}_1\,\mathbbm{1}_{\{\eta_{j_1-1}=1\}}
\qquad \text{a.s.}
\end{aligned}
\]
Also set $A^{j_2+1,j_2}_0 = A^{j_2+1,j_2}_1= 1$. We claim that
\[
A^{j_1 ,j_2}_r = \sum_{s\in\{0,1\}} M(t)_{r,s} \, A^{j_1+1 ,j_2}_s.
\]
The proof goes as follows. For $j\in\N$, let $\mathcal{G}_{j}$ be the sigma-field generated by $(\eta_{j-1},\ell_{j},\ell_{j}^{\prime})$. By the tower property, we have
\[
\begin{aligned}
A^{j_1,j_2}_r 
&= \EE\bigg[ \ee^{\langle t,h(W_{j_1})\rangle}\, 
\EE\Big[\ee^{\sum_{i=j_1+1}^{j_2} \langle t,h(W_i)\rangle} \,\Big|\, \mathcal{G}_{j_1} \Big] 
\,\Big|\, \eta_{j_1-1}=r \bigg].
\end{aligned}
\]
Since $\{(\ell_i,\ell_i^{\prime})\}_{i\in\N_{-1}}$ is i.i.d., the random variables $\{(\ell_i,\ell_i^{\prime})\}_{i\in\N_{ j_1+1}}$ are independent of $\mathcal{G}_{j_1}$. Moreover, for $i\geq j_1+1$, the quantity $h(W_i)$ depends on the past only through $\eta_{j_1}$. Hence
\[
\EE\Big[\ee^{ \sum_{i=j_1+1}^{j_2} \langle t,h(W_i)\rangle} \,\Big|\, \mathcal{G}_{j_1} \Big]
= A^{j_1+1,j_2}_{\eta_{j_1}} \qquad \text{a.s.}
\]
Therefore, for every $r\in\{0,1\}$,
\[
\begin{aligned}
A^{j_1,j_2}_r 
&= \EE\bigg[\ee^{\langle t,h(W_{j_1})\rangle}\, A^{j_1+1,j_2}_{\eta_{j_1}} \,\Big|\, \eta_{j_1-1}=r \bigg]\\
&= \sum_{s\in\{0,1\}} \EE\bigg[\ee^{\langle t,h(W_{j_1})\rangle}\, \mathbbm{1}_{\{\eta_{j_1}=s\}} 
\,\Big|\, \eta_{j_1-1}=r \bigg]\, A^{j_1+1,j_2}_s\\
&= \sum_{s\in\{0,1\}} M(t)_{r,s}\,A^{j_1+1,j_2}_s,
\end{aligned}
\]
which proves the claim. 

For each $j\in\N$, iteration of the recursion gives
\[
\begin{aligned}
A^{1,j}_r 
&= \sum_{s_{1},\ldots, s_{j}\in\{0,1\}} M(t)_{r,s_1}\,M(t)_{s_1,s_2} \times\cdots\times M(t)_{s_{j-1},s_j}\\
&= \sum_{s\in\{0,1\}} M(t)^j_{r,s} = (M(t)^j \mathbf{1})_r, \qquad r\in\{0,1\},
\end{aligned}
\]
where $(M(t)^j \mathbf{1})_r$ denotes the entry of the vector $M(t)^j \mathbf{1}$ corresponding to state $r$. Finally, since
\[
j\langle t,Z_j\rangle = \sum_{i=1}^j \langle t,h(W_i)\rangle,
\]
we get
\[
\begin{aligned}
\EE\big[\ee^{j\langle t,Z_j\rangle}\big] 
= \sum_{r\in\{0,1\}} \PP(\eta_0=r)\,A^{1,j}_r = \mu M(t)^j\mathbf{1},
\end{aligned}
\]
where $\mu = (q,p)$.
\end{proof}

Using the representation derived in Lemma~\ref{lem:matrix-rep}, we next identify the cumulant generating function associated with the process $\{Z_j\}_{j\in\N}$.

\begin{lemma}{\bf[Spectral representation of $\Lambda$]}
\label{lem:existence}
For every $t = (t_-,t_0,t_+) \in \R^3$, $\Lambda(t)$ exists in $[-\infty,\infty]$. In particular,
\[
\Lambda(t) =
\begin{cases}
\log \vartheta(t), & \text{if } t_0 < \log(1/p),\\
\infty, & \text{if } t_0\geq \log(1/p),
\end{cases}
\]
where $\vartheta(t)$ is the Perron--Frobenius eigenvalue of $M(t)$.
\end{lemma}

\begin{proof}
Fix $t = (t_-,t_0,t_+) \in \R^3$. First suppose that $t_0 < \log(1/p)$. Then all entries of $M(t)$ are finite and strictly positive. Hence $M(t)$ is irreducible. By the Perron--Frobenius theorem (see, for instance, \cite[Theorem 3.1.1]{DZ}), $M(t)$ has a Perron--Frobenius eigenvalue $\vartheta(t)>0$ such that, for every positive vector $\mathbf{v}=(v_0,v_1)^{\mathsf T} \in\R^2$, 
\[
\lim_{j\to\infty} \frac{1}{j} \log (M(t)^j\mathbf{v})_r = \log\vartheta(t), \qquad r\in\{0,1\},
\]
where $(M(t)^j\mathbf{v})_r$ denotes the entry of the vector $M(t)^j\mathbf{v}$ indexed by $r$. In particular, taking $\mathbf{v} = \mathbf{1} = (1,1)^{\mathsf T}$, we get
\[
\lim_{j\to\infty} \frac{1}{j} \log (M(t)^j\mathbf {1})_r = \log\vartheta(t), \qquad r\in\{0,1\}.
\]
Since $\mu=(q,p)$ has strictly positive coordinates, we have
\[
q (M(t)^j\mathbf{1})_0 \leq \mu M(t)^j\mathbf{1} 
\leq 2\max\big\{(M(t)^j\mathbf{1})_0 , (M(t)^j\mathbf{1})_1\}.
\]
Therefore, by Lemma~\ref{lem:matrix-rep},
\[
\Lambda(t) = \lim_{j\to\infty} \frac{1}{j} \log \mu M(t)^j\mathbf{1} = \log\vartheta(t).
\]

Next, suppose that $t_0 \geq \log(1/p)$. Since $0 \leq h^- \leq 2$ and $0 \leq h^+\leq 2$, we have
\[
\sum_{i=1}^j \langle t,h(W_i)\rangle \geq t_0\, h^0(W_1) - 2j\,(|t_-| + |t_+|).
\]
Hence, noting also that $h^0(W_1)\geq \ell_1-2$, we get
\[
\EE\big[\ee^{j\langle t,Z_j\rangle}\big] \geq \ee^{-2j(|t_-|+|t_+|)}\, \EE\big[\ee^{t_0\,h^0(W_1)}\big] 
\geq \ee^{-2j(|t_-|+|t_+|)}\,\EE\big[\ee^{t_0(\ell_1-2)}\big] = \infty.
\]
Therefore $\Lambda(t) = \infty$.
\end{proof}


\paragraph{Step 6.\ Effective domain and regularity properties.}

By Lemma~\ref{lem:existence},
\[
\cD_{\Lambda} = \{t = (t_- , t_0 , t_+) \in \R^3\colon\, t_0 < \log(1/p)\}
\]
is the effective domain of $\Lambda$, which is an open set. Hence
\[
(0,0,0) \in \mathrm{int}(\cD_{\Lambda}) = \cD_{\Lambda}.
\]
Moreover, for every $t\in \cD_{\Lambda}$, all entries of $M(t)$ are finite, strictly positive and analytic functions of $t$. Since $M(t)$ is a $2 \times 2$ matrix, its Perron--Frobenius eigenvalue is given by
\[
\vartheta(t) = \frac{\mathrm{tr}(M(t)) + \sqrt{(\mathrm{tr}(M(t)))^2 - 4\,\mathrm{det}(M(t))}}{2},
\]
where $\mathrm{tr}(M(t))$ and $\mathrm{det}(M(t))$ denote the trace and the determinant of $M(t)$, respectively. Since
\[
(\mathrm{tr}(M(t)))^2 - 4\,\mathrm{det}(M(t)) = \big(M(t)_{0,0}-M(t)_{1,1}\big)^2 + 4\, M(t)_{0,1}\,M(t)_{1,0} > 0,
\]
it follows that $\vartheta$ is analytic on $\cD_{\Lambda}$. Since $\Lambda(t) = \log \vartheta(t)$ and $\vartheta(t)>0$ for $t \in \cD_{\Lambda}$, it follows that $\Lambda$ is analytic on $\cD_{\Lambda}$.

We next establish the regularity properties of $\Lambda$ that are needed for the application of the G\"artner-Ellis theorem.

\begin{lemma}{\bf[Smoothness and steepness of $\Lambda$]}
\label{lem:regularity}
The cumulant generating function $\Lambda$ satisfies the following:
\begin{itemize}
\item[{\rm (a)}] $\Lambda$ is differentiable on $\cD_{\Lambda}$. 
\item[{\rm (b)}] $\Lambda$ is lower semi-continuous on $\R^3$. 
\item[{\rm (c)}] $\Lambda$ is steep at $\partial\cD_{\Lambda}$, i.e.,
\[
\lim_{t \to \partial\cD_{\Lambda}\colon\, t\in \cD_{\Lambda}}\|\nabla \Lambda(t)\|_2 = \infty,
\]
where $\nabla$ denotes the gradient operator and $\|\cdot\|_2$ denotes the Euclidean norm.
\end{itemize}
\end{lemma}

\begin{proof}
(a) Since $\Lambda$ is analytic on $\cD_{\Lambda}$, it is differentiable on $\cD_{\Lambda}$.

\medskip\noindent
(b) Fix $t = (t_- , t_0 , t_+)\in \R^3$. Let $t_n \to t$, as $n\to\infty$. We claim that
\[
\Lambda(t) \leq \liminf_{n\to\infty} \Lambda(t_n).
\]
If $t \in \cD_{\Lambda}$, then $\Lambda$ is continuous at $t$, and the claim follows. Now suppose that $t \notin \cD_{\Lambda}$. Then $t_0 \geq -\log p$ and $\Lambda(t) = \infty$. Write $t_n = (t_{n,-},t_{n,0},t_{n,+})$. We consider the following two cases separately:

\medskip\noindent
\textit{Case 1}: Let $t_0 > \log(1/p)$. Choose $\varepsilon>0$ such that
\[
t_0 > \log(1/p) + \varepsilon.
\]
Since $t_n\to t$ as $n\to\infty$, there exists $N_{\varepsilon} \in \N$ such that, for all $n\geq N_{\varepsilon}$,
\[
t_{n,0} > t_0 - \tfrac12 \varepsilon > \log(1/p) + \tfrac12 \varepsilon.
\]
Hence $t_n \notin \cD_{\Lambda}$ for all $n\geq N_{\varepsilon}$. Therefore $\Lambda(t_n)=\infty$ for all $n \geq N_{\varepsilon}$. Consequently,
\[
\liminf_{n\to\infty} \Lambda(t_n) \geq \inf_{n\geq N_{\varepsilon}} \Lambda(t_n) = \infty = \Lambda (t),
\]
which verifies the claim in this case.

\medskip\noindent
\textit{Case 2}: Let $t_0 = \log(1/p)$. If $t_n\notin\cD_{\Lambda}$ eventually, then the claim follows as above. Otherwise, there exists a subsequence $\{t_{n_k}\}_{k\in\N}$ such that
\[
\liminf_{n\to\infty} \Lambda(t_n) = \lim_{k\to \infty} \Lambda(t_{n_k}),
\]
and $t_{n_k}\in\cD_{\Lambda}$ for all $k\in\N$. Since 
\[
\lim_{k\to\infty} t_{n_{k},0} = \lim_{n\to\infty} t_{n,0} = \log(1/p),
\]
we have 
\[
\lim_{k\to\infty}(1- p\,\ee^{t_{n_k,0}}) = 0.
\]
Hence $\lim_{k\to\infty} M(t_{n_k})_{0,1} = \infty$ and $\lim_{k\to\infty} M(t_{n_k})_{1,1} = \infty$. Consequently, $\lim_{k\to\infty} \vartheta(t_{n_k}) = \infty$, and therefore
\[
\lim_{k\to\infty} \Lambda(t_{n_k}) = \infty.
\]
Hence $\liminf_{n\to\infty} \Lambda(t_n) = \infty = \Lambda(t)$, which verifies the claim in this case. Since $t\in \R^3$ is arbitrary, the claim implies that $\Lambda$ is lower semi-continuous on $\R^3$.

\medskip\noindent
(c) For $t = (t_-,t_0,t_+) \in \cD_{\Lambda}$, set
\[
C(t) = \big(M(t)_{0,0}-M(t)_{1,1}\big)^2 + 4M(t)_{0,1}M(t)_{1,0}.
\]
Then
\[
\frac{\partial\Lambda(t)}{\partial t_0} = \frac{1}{2\vartheta(t)} \left[A(t)+\frac{B(t)}{\sqrt{C_(t)}}\right],
\]
where
\[
A(t)=\frac{\partial}{\partial t_0}\mathrm{tr}(M(t))
\]
and
\[
B(t) = \big(M(t)_{0,0}-M(t)_{1,1}\big)\,\frac{\partial}{\partial t_0} \big(M(t)_{0,0}-M(t)_{1,1}\big) 
+ 2M(t)_{1,0}\,\frac{\partial}{\partial t_0}M(t)_{0,1}.
\]
Set $\alpha_{t_0} = 1-p\,\ee^{t_0}$ and $\alpha_{t_-,t_+} = q\,p^2\,\ee^{t_- + 2 t_+}$. As $t \to \partial\cD_{\Lambda}$, we have $t_0 \to \log(1/p)$ and therefore $\alpha_{t_0}\to 0$. In this case, 
\[
M(t)_{1,1} \sim \frac{\alpha_{t_-,t_+}}{\alpha_{t_0}},
\qquad
M(t)_{0,1} \sim \frac{\ee^{t_-}\alpha_{t_-,t_+}}{\alpha_{t_0}},
\qquad
M(t)_{0,0}-M(t)_{1,1} \sim -\frac{\alpha_{t_-,t_+}}{\alpha_{t_0}}.
\]
Hence
\[
C(t) \sim \frac{\alpha_{t_-,t_+}^2}{\alpha_{t_0}^2},
\qquad
\sqrt{C_t}\sim \frac{\alpha_{t_-,t_+}}{\alpha_{t_0}}.
\]
Since $\vartheta(t) = \tfrac12[\mathrm{tr}(M(t)) + \sqrt{C(t)}]$ and  $\mathrm{tr}(M(t)) \sim \frac{\alpha_{t_-,t_+}}{\alpha_{t_0}}$, we obtain
\[
\vartheta(t) \sim \frac{\alpha_{t_-,t_+}}{\alpha_{t_0}}.
\]
Furthermore, since $p\,\ee^{t_0}\sim 1$, we have
\[
A(t) = q^2\ee^{t_0}+\frac{p\,\ee^{t_0}\,\alpha_{t_-,t_+}}{\alpha_{t_0}^2}
\sim \frac{\alpha_{t_-,t_+}}{\alpha_{t_0}^2}
\]
and
\[
B(t) = \big(M(t)_{0,0}-M(t)_{1,1}\big)\,\left[q^2\ee^{t_0}-\frac{p\,\ee^{t_0}\,\alpha_{t_-,t_+}}{\alpha_{t_0}^2}\right] 
+ \frac{2\,q\,p\,\ee^{t_- +t_0}\,\alpha_{t_-,t_+}}{\alpha_{t_0}^2}
\sim \frac{\alpha_{t_-,t_+}^2}{\alpha_{t_0}^3}.
\]
Therefore
\[
A(t) + \frac{B(t)}{\sqrt{C(t)}} \sim \frac{2\alpha_{t_-,t_+}}{\alpha_{t_0}^2}.
\]
Combining the above estimates, we get
\[
\frac{\partial\Lambda(t)}{\partial t_0} \sim \frac{1}{\alpha_{t_0}}.
\]
Since $\alpha_{t_0}\to 0$, we get
\[
\|\nabla\Lambda(t)\|_2 \geq \left|\frac{\partial\Lambda(t)}{\partial t_0} \right| \to\infty,
\]
which settles the claim.
\end{proof}


\paragraph{Step 7.\ LDP for the empirical type averages.}

Steps 5--7 allow us to derive the LDP for the sequence of empirical type averages.

\begin{lemma}{\bf [Target LDP]}
\label{lem:LDP-type-fractions}
\begin{itemize}
\item[{\rm (a)}] 
The family of random variables $\{Z_j\}_{j\in\N}$ satisfies the LDP on $F_0$ with rate $j$ and with good rate function
\[
\Lambda^*(\theta) = \sup_{t\in\R^3} \{\langle t,\theta\rangle-\Lambda(t)\}, \qquad \theta\in F_0,
\]
which is the Fenchel--Legendre transform of $\Lambda$.
\item[{\rm (b)}] 
For every $\theta\in F_0 = [0,2] \times [0,\infty) \times [0,2]$,
\[
\Lambda^*(\theta) = \inf_{\nu\in G^{-1}(\theta)} I(\nu),
\]
where $G^{-1}(\theta)$ denotes the pre-image set $\{\nu\in\mathcal{P}(E)\colon\, G(\nu)=\theta\}$.
\end{itemize}
\end{lemma}

\begin{proof}
(a) The claim follows from Lemma~\ref{lem:regularity} and the G\"artner--Ellis theorem \cite[Theorem V.6]{FdH}.

\medskip\noindent
(b) Define
\[
\bar{I}(\theta) = \left\{\begin{array}{ll}
\inf_{\nu\in G^{-1}(\theta)} I(\nu), &\theta\in F_0,\\
\infty, &\theta \notin F_0.
\end{array}
\right.
\]
Also define $\Lambda^*(\theta)=\infty$ for $\theta \notin F_0$. We need to show that $\bar{I}(\theta) = \Lambda^*(\theta)$ for all $\theta \in \R^3$. Since $\R^3$ equipped with the Euclidean topology is a locally convex Hausdorff topological vector space, and the dual space of $\R^3$ is canonically isomorphic to $\R^3$, the following duality argument can be applied.

Since $\nu \mapsto I(\nu)$ is convex and $\nu \mapsto G(\nu)$ is linear, it follows that $\theta \mapsto \bar{I}(\theta)$ is convex on $\R^3$. Let $\nu^* = (\zeta^{\otimes 2})^{\otimes 3}$, and set $\theta^* = G(\nu^*)$. Since $I(\nu^*) = 0$, we have
\[
\bar{I}(\theta^*) \leq I(\nu^*) = 0.
\]
Hence $\bar{I}(\theta^*) < \infty$, and therefore $\bar{I} \not\equiv \infty$. In Section~\ref{sec:ProofLDP2} we show that $\bar{I}$ is lower semi-continuous (see Remark~\ref{rem:lsc}). 

Next, consider the Fenchel--Legendre transform of $\bar{I}$ given by
\[
\bar{I}^*(t) = \sup_{\theta \in \R^3} \big[\langle t,\theta\rangle- \bar{I}(\theta)\big]
= \sup_{\theta \in F_0} \big[\langle t,\theta\rangle- \bar{I}(\theta)\big], \qquad t \in \R^3.
\]
To complete the proof, it remains to show that
\[
(\spadesuit) \qquad \bar{I}^*(t) = \Lambda(t), \qquad t \in \R^3.
\]
Indeed, once this identity is established, the Fenchel--Moreau theorem \cite[Theorem 1.11]{Brezis} gives
\[
\bar{I}(\theta) = \bar{I}^{**}(\theta) = \Lambda^*(\theta), \qquad \theta\in F_0,
\]
where $\bar{I}^{**}$ denotes the Fenchel--Legendre transform of $\bar{I}^*$. This is in accordance with the fact that the rate function in an LDP is necessarily unique \cite[Chapter III]{FdH}. 

\medskip\noindent
{\bf Proof of $(\spadesuit)$.} For $t = (t_- , t_0 , t_+) \in \R^3$, we have
\[
\bar{I}^*(t) = \sup_{\theta \in F_0} \big[\langle t,\theta\rangle - \inf_{\nu\in G^{-1}(\theta)} I(\nu)\big]
= \sup_{\theta \in F_0} \sup_{\nu\in G^{-1}(\theta)} \big[\langle t,G(\nu)\rangle- I(\nu)\big]
= \sup_{\nu\in\mathcal{P}(E)} \big[\langle t,G(\nu)\rangle - I(\nu)\big].
\]
Since $E$ is endowed with the discrete topology, each function $h^\chi$ is continuous on $E$. Hence, for $\chi\in\{-,+\}$ the functionals $G^{\chi}$ are continuous with respect to the weak topology, because the integrands $h^\chi$ are bounded and continuous functions. In contrast, $h^0$ is unbounded, so $G^0$ is not continuous. For $M \in \N$, define the truncated function $h^0_M = h^0 \wedge M$ and the corresponding functional
\[
G_{M}^{0}(\nu) = \int_{E} h^0_M\, \mathrm{d}\nu, \qquad \nu\in\mathcal{P}(E).
\]
Then $G^{0}_{M}$ is continuous on $\mathcal{P}(E)$. For $\nu\in\mathcal{P}(E)$, set
\[
G_{M}(\nu) = \big(G^{-}(\nu),G_{M}^{0}(\nu),G^{+}(\nu)\big).
\]
Since $G_{M}$ is continuous and non-decreasing in $M$, the map 
\[
\nu\mapsto G_{M}(\nu)
\]
is continuous from $\mathcal{P}(E)$ into $F_{M} = [0,2] \times [0,M] \times [0,2]$, and non-decreasing in $M$.

\medskip\noindent
\underline{Case $t_0 \leq 0$}. For $t_0=0$, $G^0$ drops out and and Varadhan's Lemma gives
\[
\Lambda(t) = \bar{I}^*(t).
\]
For $t_0<0$, the map $\nu \mapsto \langle t , G(\nu)\rangle$ is upper semi-continuous under the weak topology, because $G^-$, $G^+$ are continuous and $G^0$ is lower semi-continuous (by the Portmanteau theorem, because it is the integral of the non-negative continuous function $h^0$). Moreover, since $G^\chi \leq 2$ for $\chi \in\{-,+\}$, we have $\langle t,G(L_j)\rangle \leq  2 |t_-| + 2|t_+|$. Therefore, for $M > 2 |t_-| + 2|t_+|$,
\[
\EE\big[\ee^{j \langle t,G(L_j)\rangle}\, \mathbbm{1}_{\{\langle t,G(L_j)\rangle \geq M\}}\big] = 0.
\]
Consequently,
\[
\lim_{M\to\infty}\limsup_{j\to\infty}\frac{1}{j}\log\EE\Big[\ee^{j \langle t,G(L_j)\rangle}\, 
\mathbbm{1}_{\{\langle t,G(L_j)\rangle \geq M\}}\Big] = -\infty,
\]
which, by \cite[Lemma 4.3.6]{DZ}, implies that
\[
\Lambda(t) \leq \bar{I}^*(t).
\]
To derive the reverse bound, we use the Donsker--Varadhan variational formula for the relative entropy \cite[Lemma 1.4.3(a) and Appendix C]{DE}. It suffices to show that
\[
\langle t, G(\nu) \rangle - I(\nu) \leq \Lambda(t) \qquad \forall\,\nu\in\mathcal{P}(E) \text{ admissible}.
\]

Let $\nu\in\mathcal{P}(E)$ be an admissible measure. From the Donsker--Varadhan variational formula it follows that
\[
I(\nu) = H\big(\nu ~\big\|~ \pi_{1,2}\nu \otimes \zeta^{\otimes 2}\big)
= \sup_{\psi\in \mathcal{B}(E)} \left\{ \int_{E} \psi\, \mathrm{d}\nu
- \log \int_{E} \mathrm{e}^{\psi}\,\mathrm{d}(\pi_{1,2}\nu \otimes \zeta^{\otimes 2})\right\},
\]
where $\mathcal{B}(E)$ denotes the space of bounded measurable functions $\psi\colon\,E\to\R$. Hence
\[
(\clubsuit) \qquad \int_{E} \psi\, \mathrm{d}\nu - I(\nu) 
\leq \log \int_{E} \mathrm{e}^{\psi}\,\mathrm{d}(\pi_{1,2}\nu \otimes \zeta^{\otimes 2}) \qquad \forall\,\psi\in \mathcal{B}(E).
\]
Let $\mathbf{u}=(u_0,u_1)^{\mathsf T}$ be the right eigenvector associated with the Perron--Frobenius eigenvalue $\ee^{\Lambda(t)}$, whose coordinates are strictly positive. Such an eigenvector exists by the Perron--Frobenius theorem (see \cite[Theorem 3.1.1(c)]{DZ}). Now, for $x=((a,a^\prime),(b,b^\prime),(c,c^\prime))\in E$, set
\[
f(x) = \langle t , h(x)\rangle + \log u_{\eta_1(x)} - \log u_{\eta_0(x)},
\]
where
\[
\eta_0(x) = \mathbbm{1}_{\{b>0\}},\qquad
\eta_1(x) = \mathbbm{1}_{\{c>0\}}.
\]
Since $\pi_{1,2}\nu = \pi_{2,3}\nu$, we have $\pi_{2}\nu = \pi_{3}\nu$, which implies that $\nu(\eta_0 =r) = \nu(\eta_1 =r)$ for all $r\in\{0,1\}$.
Consequently,
\[
\int_{E} \log u_{\eta_0}\, \mathrm{d}\nu = \int_{E} \log u_{\eta_1}\, \mathrm{d}\nu,
\]
and hence
\[
\int_{E} f\, \mathrm{d}\nu = \int_{E} \langle t , h \rangle\, \mathrm{d}\nu = \langle t , G(\nu) \rangle .
\]
Moreover, by Tonelli's theorem,
\[
\begin{aligned}
\int_{E} \mathrm{e}^{f}\,\mathrm{d}(\pi_{1,2}\nu \otimes \zeta^{\otimes 2}) 
&= \int_{E} \mathrm{e}^{\langle t,h \rangle}\,\frac{u_{\eta_1}}{u_{\eta_0}}\,
\mathrm{d}(\pi_{1,2}\nu \otimes \zeta^{\otimes 2})\\
&= \int_{(\N_0^2)^2}\frac{1}{u_{\eta_0}}\bigg(\int_{\N_0^2} \mathrm{e}^{\langle t,h\rangle}\, 
u_{\eta_1}\,\mathrm{d}\zeta^{\otimes 2}\bigg)\mathrm{d}\pi_{1,2}\nu \\
&= \int_{(\N_0^2)^2}\frac{1}{u_{\eta_0}}\bigg(\sum_{s\in\{0,1\}}u_s \int_{\N_0^2} 
\mathrm{e}^{\langle t , h \rangle}\, \mathbbm{1}_{\{\eta_1=s\}}\,
\mathrm{d}\zeta^{\otimes 2}\bigg)\mathrm{d}\pi_{1,2}\nu \\
&= \int_{(\N_0^2)^2}\frac{1}{u_{\eta_0}}\sum_{s\in\{0,1\}}u_s \,M(t)_{\eta_0 , s}\,\mathrm{d}\pi_{1,2}\nu,
\end{aligned}
\]
where in the last line we use the definition of $M(t)$ and the fact that $h(x)$ depends on $((a,a'),(b,b'))$ only through $\eta_0(x)$, so that the inner integral in the penultimate line is the conditional expectation defining $M(t)_{\eta_0,s}$. Now, using the eigenvalue equation $M(t)\,\mathbf{u} = \ee^{\Lambda(t)}\,\mathbf{u}$, we have 
\[
\sum_{s\in\{0,1\}}M(t)_{r,s}\,u_s = \ee^{\Lambda(t)}\,u_r, \qquad r\in\{0,1\}.
\]
Consequently, 
\[
\int_{E} \mathrm{e}^{f}\,\mathrm{d}(\pi_{1,2}\nu \otimes \zeta^{\otimes 2}) = \ee^{\Lambda(t)}.
\]
Finally, take $\psi = \psi_N = f\vee(-N)$, $N\in\N$, in $(\clubsuit)$. Since $t_0 < 0$, the function $f$ is bounded from above, say $ f\leq C$. Letting $N\to\infty$, and applying the monotone convergence theorem to $C-\psi_N$ and the dominated convergence theorem to $\ee^{\psi_N}$, we get that
\[
\langle t,G(\nu) \rangle - I(\nu) = \int_{E} f\, \mathrm{d}\nu - I(\nu) 
\leq \log \int_{E} \mathrm{e}^{f}\,\mathrm{d}(\pi_{1,2}\nu \otimes \zeta^{\otimes 2}) = \Lambda(t),
\]
which implies that
\[
\bar{I}^*(t) \leq \Lambda(t),
\]
and completes the proof of $(\spadesuit)$ in this case.

\medskip\noindent
\underline{Case $0 < t_0 < \log(1/p)$}. For $M\in\N$, define
\[
\Lambda_M(t) = \sup_{\nu\in\mathcal{P}(E)} \big\{\langle t,G_M(\nu)\rangle - I(\nu)\big\}.
\]
Since $G_M$ is bounded and continuous, Varadhan's Lemma \cite[Theorem~III.13]{FdH} yields
\[
\Lambda_M(t) = \lim_{j\to\infty}\frac{1}{j} \log\EE\big[\ee^{j\langle t,G_M(L_j)\rangle} \big].
\]
Moreover, since $t_0>0$ and $G_M^0 \leq G^0$,
\[
(\blacksquare) \qquad \Lambda_M(t) \leq \sup_{\nu\in\mathcal{P}(E)} 
\big\{ \langle t,G(\nu)\rangle - I(\nu) \big\} = \bar{I}^*(t), \qquad \forall\, M\in\N.
\]
Let $M_M(t)$ be the matrix obtained after replacing $h^0$ by $h_M^0$ in the definition of the matrix $M(t)$ introduced before Lemma~\ref{lem:matrix-rep}. The arguments of Lemmas~\ref{lem:matrix-rep} and \ref{lem:existence} remain valid with $h^0$ replaced by $h_M^0$. Hence
\[
\Lambda_M(t)=\log\vartheta(M_M(t)).
\]
Since $t_0 < \log(1/p)$, all entries of $M(t)$ are finite. Moreover, $h_M^0 \uparrow h^0$ as $M \to \infty$, and therefore, by the monotone convergence theorem,
\[
M_M(t)_{r,s} \uparrow M(t)_{r,s}, \qquad M\to\infty \qquad r,s \in \{0,1\}.
\]
Since the Perron--Frobenius eigenvalue depends continuously on the matrix entries, we have
\[
\vartheta(M_M(t)) \to \vartheta(M(t)), \qquad M \to \infty,
\]
where $\vartheta(M_M(t))$ is the Perron--Frobenius eigenvalue of $M_M(t)$. Consequently,
\[
\Lambda_M(t) = \log\vartheta(M_M(t)) \to \log\vartheta(t) = \Lambda(t), \qquad M\to\infty.
\]
Hence, from $(\blacksquare)$,
\[
\Lambda(t) \leq \bar{I}^*(t).
\]

The reverse bound is easy. Indeed, by Varadhan's Lemma applied to the bounded continuous functional $G_M$ for $M \in \N$, in combination with $t_0 > 0$ and $G^0(\nu) = \sup_{M\in\N} G^0_{M}(\nu)$ for all $\nu \in \mathcal{P}(E)$ by the monotone convergence theorem, we also have
\[
\begin{aligned}
\Lambda(t)
&\geq \sup_{M\in\N} \sup_{\nu\in\mathcal{P}(E)} \big[\langle t, G_{M}(\nu) \rangle - I(\nu)\big]\\
&= \sup_{\nu\in\mathcal{P}(E)} \sup_{M\in\N} \big[\langle t, G_{M}(\nu) \rangle - I(\nu)\big]
= \sup_{\nu\in\mathcal{P}(E)} \big[\langle t, G(\nu) \rangle - I(\nu)\big] = \bar{I}^*(t),
\end{aligned}
\]
and so $(\spadesuit)$ holds.

\medskip\noindent
\underline{Case $t_0 \geq \log(1/p)$}. In this case we know from Lemma~\ref{lem:existence} that $\Lambda(t) = \infty$, and so it remains to show that also $\bar{I}^*(t) = \infty$. First assume that $t_0 > \log(1/p)$. Let $n \in \N_2$ be arbitrary, and define the measure
\[
\nu_n = \delta_{((n,0),(n,0),(n,0))} \in \mathcal{P}(E).
\]
Note that $\nu_n$ is admissible because $\pi_{1,2}\nu_n = \pi_{2,3}\nu_n = \delta_{((n,0),(n,0))}$. Therefore
\[
\begin{aligned}
I(\nu_n) &= H(\nu_n ~\|~ \pi_{1,2}\nu_n \otimes \zeta^{\otimes 2})\\
& = -\log\, (\pi_{1,2}\nu_n \otimes \zeta^{\otimes 2})\big(\big\{(n,0),(n,0),(n,0)\big\}\big)\\
& = -\log \big(\zeta(n)\,\zeta(0)\big) \\
& = -n\log p - 2\log q.
\end{aligned}
\]
Noting that $G(\nu_n)=(1,n-2,2)$, we have
\[
\bar{I}^*(t) \geq \langle t,G(\nu_n)\rangle - I(\nu_n) = n\,(t_0 +\log p) 
+ 2\left(\tfrac12 t_- - t_0 + t_+ + \log q \right) \qquad \forall\, n \in \N_2.
\]
Since $t_0 > \log(1/p)$, letting $n\to\infty$, we get $\bar{I}^*(t)=\infty$, and so $(\spadesuit)$ holds.

Next, assume that $t_0 = \log(1/p)$. Suppose, by contradiction, that $\bar{I}^*(t) < \infty$. Choose $s = (t_- ,t_0 -1 ,t_+)$. Then $s$ belongs to the effective domain of $\bar{I}^*$ because $\bar{I}^*(s) = \Lambda(s) < \infty$. For $\lambda\in (0,1)$, set
\[
t_{\lambda} = (1-\lambda)\, t + \lambda\, s.
\]
Then $t_{\lambda}$ also belongs to the effective domain of $\bar{I}^*$. On the other hand, since $\bar{I}^*$ is convex and $\bar{I}^*(t_\lambda) = \Lambda(t_\lambda)$, we have
\[
\Lambda(t_{\lambda}) \leq (1-\lambda)\,\bar{I}^*(t)+\lambda\, \bar{I}^*(s).
\]
The right-hand side remains bounded as $\lambda \downarrow 0$ because both $\bar{I}^*(t)$ and $\bar{I}^*(s)$ are finite. However, $\Lambda(t_{\lambda}) = \log \vartheta(t_{\lambda}) \to \infty$ as $\lambda \downarrow 0$, which leads to a contradiction. Therefore $\bar{I}^*(t)=\infty$, and so $(\spadesuit)$ holds in this case as well.
\end{proof}


\paragraph{Step 8.\ Completion of the proof of Theorem \ref{LDP1}.}

By Lemma~\ref{lem:LDP-type-fractions}, the sequence of random elements $\{Z_j\}_{j\in\N}$ satisfies the LDP with values in the Polish space $F_0$, with rate $j$, and with good rate function $\Lambda^*$. Recalling from Step 4 that 
\[
Q_j = Z_j + \tfrac{1}{j} R,\qquad j \in \N_4,
\]
with $R$ a uniformly bounded remainder term, we conclude that $\{Q_j\}_{j\in\N_4}$ and $\{Z_j\}_{j\in\N_4}$ are exponentially equivalent and hence share the same LDP by \cite[Theorem~4.2.13]{DZ}.  Applying the homeomorphism $\Psi$ introduced in Step 4, which yields that $\Upsilon_j = \Psi(Q_j)$, and using the contraction principle \cite[Theorem~III.20]{FdH}, we obtain that the sequence $\{\Upsilon_j\}_{j\in\N_4}$ satisfies the LDP with values in the Polish space $(0,\infty) \times \mathcal{S}_2$, with rate $j$, and with good rate function
\[
J(x,s) = \inf_{\theta \in \Psi^{-1}(x,s)} \Lambda^*(\theta) = \Lambda^*(xs) = \inf_{\nu\in G^{-1}(xs)} I(\nu).
\]
Since $\|s\|_1 = 1$ for all $s\in\mathcal{S}_2$, the constraint $G(\nu) = x s$ is equivalent to $\|G(\nu)\|_1 = x$ and $G(\nu)/\|G(\nu)\|_1 = s$. Hence the rate function derived above coincides with
\[
J(x,s) = \inf_{\nu\in \Phi(x,s)} I(\nu),
\]
where
\[
\Phi(x,s) = \big\{\nu \in \mathcal{P}(E)\colon\,\|G(\nu)\|_1 = x,\,G(\nu)/\|G(\nu)\|_1 = s\big\}.
\]
This completes the proof of Theorem~\ref{LDP1}.
\end{proof}


\subsection{Proof of the second LDP theorem} 
\label{sec:ProofLDP2}

In this section we prove Theorem~\ref{LDP2}. 

\begin{proof}
(a) Write 
\[
\Phi(x,s) = \{\nu \in \mathcal{P}(E)\colon\,\|G(\nu)\|_1 = x,\,G(\nu) = xs\}.
\] 
Fix $s \in \mathcal{S}_2$. Pick $x_1, x_2 \in (0,\infty)$ with $x_1 \neq x_2$ and $\alpha \in (0,1)$. Note that if $\nu_1 \in \Phi(x_1,s)$ and $\nu_2 \in \Phi(x_2,s)$, then $\alpha \nu_1 + (1-\alpha) \nu_2 \in \Phi(\alpha x_1 + (1-\alpha) x_2,s)$, because $\nu \mapsto \|G(\nu)\|_1$ and $\nu \mapsto G(\nu)$ are linear. Hence
\[
\begin{aligned}
J(\alpha x_1 + (1-\alpha) x_2,s) 
&\leq \inf_{ {\nu_1 \in \Phi(x_1,s)} \atop {\nu_2 \in \Phi(x_2,s)} } I(\alpha \nu_1 + (1-\alpha) \nu_2)\\
&\leq \inf_{ {\nu_1 \in \Phi(x_1,s)} \atop {\nu_2 \in \Phi(x_2,s)} } [\alpha I(\nu_1) + (1-\alpha) I(\nu_2)]
= \alpha J(x_1,s) + (1-\alpha) J(x_2,s), 
\end{aligned}
\]
where the first inequality and the equality use the variational formula for $J$, while the second inequality uses the convexity of $\nu \mapsto I(\nu)$. Hence $J$ is convex in its first coordinate. A similar argument shows that $J$ is convex in its second coordinate. 

Note that $I(\nu) = 0$ if and only if $\nu = \pi_{1,2} \nu \otimes \zeta^{\otimes 2}$. Writing $\pi_j\nu$ for the $j$th marginal of $\nu$ on $\N_0^2$, and using the consistency relation $\pi_{1,2}\nu=\pi_{2,3}\nu$, we get $\nu = \pi_1\nu \otimes (\zeta^{\otimes 2})^{\otimes 2}$. Because $\pi_{1,2}\nu = \pi_{2,3}\nu$, it in turn follows that $\nu = (\zeta^{\otimes 2})^{\otimes 3} = \nu^*$.

A straightforward computation, based on the explicit form of the functions $h^\chi$ and $G^\chi$ and the formula for $\zeta$, gives
\[
\begin{aligned}
&G^-(\nu^*) = p[1+(1-p)+(1-p)^2], \quad G^0(\nu^*) = \frac{p}{1-p} -(p+p^2) + (1-p)^3, \quad G^+(\nu^*) = p + p^2,\\
&G^-(\nu^*) + G^0(\nu^*) + G^+(\nu^*) = \frac{p}{1-p} + 1 = \frac{1}{1-p}. 
\end{aligned}
\] 
Since 
\[
x^* = G^-(\nu^*) + G^0(\nu^*) + G^+(\nu^*), \qquad 
s^* = \frac{(G^-(\nu^*),G^0(\nu^*),G^+(\nu^*))}{G^-(\nu^*) + G^0(\nu^*) + G^+(\nu^*)},
\] 
this identifies the zero of the rate function as claimed.

\medskip\noindent
(b) The proof proceeds in 5 Steps. At the end of Step 5 we find how to compute the rate function on $\mathrm{int}(\cD)$, the interior of its effective domain. 

\medskip\noindent
\emph{Step 1.}
First, we carry out the infimum of $I(\nu)$ over $\nu \in \Phi(x,s)$ with the help of the technique of Lagrange multipliers. Abbreviate $\alpha = (a,a')$, $\beta = (b,b')$, $\gamma = (c,c')$. For $\nu\in\mathcal{P}(E)$ such that $\pi_{1,2}\nu=\pi_{2,3}\nu$, the rate function $I(\nu)$ can be written as 
\[
I(\nu) = \sum_{\alpha,\beta,\gamma \in \N_0^2} \nu_{\alpha\beta\gamma} \log \nu_{\alpha\beta\gamma}
- \sum_{\alpha,\beta \in \N_0^2} (\pi_{1,2}\nu)_{\alpha\beta} \log (\pi_{1,2}\nu)_{\alpha\beta}
- \sum_{\gamma\in\N_0^2} (\pi_3\nu)_\gamma \log \zeta_{\gamma},
\]
where the variables appear as indices, and we abbreviate $\zeta_\gamma = \zeta(c)\,\zeta(c')$. Define the Lagrangian
\[
\mathcal{L}(\nu) = I(\nu) + \lambda (|\nu| - 1) 
+ \sum_{\alpha,\beta \in \N_0^2} \lambda^{\alpha\beta} \big[(\pi_{1,2}\nu)_{\alpha\beta} - (\pi_{2,3}\nu)_{\alpha\beta}\big]
+ \sum_{\chi \in S} \lambda^\chi\, (G^\chi(\nu) - xs^\chi),
\]
where $\lambda$, $(\lambda^{\alpha\beta})_{\alpha,\beta\in\N_0^2}$ and $(\lambda^\chi)_{\chi \in S}$ are Lagrange multipliers that must be chosen such that the constraints $|\nu| = \int _E \dd\nu = 1$, $\pi_{1,2}\nu = \pi_{2,3}\nu$ and $G(\nu) = xs$ in $\Phi(x,s)$ are met, respectively. We do not need a Lagrange multiplier for the constraint $\sum_{\chi \in S} G^\chi(\nu) = x$, which is automatic because $\sum_{\chi \in S} s^\chi = 1$. Having chosen the Lagrangian, for fixed $\alpha,\beta,\gamma\in\N_0$ we compute 
\[
\frac{\partial \mathcal{L}(\nu)}{\partial \nu_{\alpha\beta\gamma}} 
= \log\left(\frac{\nu_{\alpha\beta\gamma}}{(\pi_{1,2}\nu)_{\alpha\beta}\,\zeta_\gamma}\right) 
+ \lambda + (\lambda^{\alpha\beta} - \lambda^{\beta\gamma}) + \sum_{\chi \in S} \lambda^\chi h^\chi_{\alpha\beta\gamma}. 
\]
Setting this equal to zero, we get
\[
\nu_{\alpha\beta\gamma} = (\pi_{1,2}\nu)_{\alpha\beta}\, 
\ee^{- \lambda - (\lambda^{\alpha\beta} - \lambda^{\beta\gamma}) - \sum_{\chi \in S} \lambda^\chi h^\chi_{\alpha\beta\gamma}} \zeta_\gamma.
\]
Next, use that $h^\chi_{\alpha\beta\gamma}$ does not depend on $\alpha$, i.e., $h^\chi_{\beta\gamma} = h^\chi_{\alpha\beta\gamma}$. Therefore, summing over $\gamma$, we obtain
\[
1 = \sum_{\gamma\in\N_0^2} \frac{\nu_{\alpha\beta\gamma}}{(\pi_{1,2}\nu)_{\alpha\beta}} = \sum_{\gamma\in\N_0^2}
\mathrm{e}^{-\lambda-(\lambda^{\alpha\beta}-\lambda^{\beta\gamma})-\sum_{\chi\in S}\lambda^{\chi} h^{\chi}_{\beta\gamma}}\,\zeta_{\gamma} \qquad \forall\,\alpha,\beta \in \N_0^2.
\]
Hence $\lambda^{\alpha\beta}$ does not depend on $\alpha$. Put $\lambda^{\alpha\beta} = \lambda^\beta$. Then, the last equation reads
\[ 
(\square) \qquad \lambda^\beta = - \lambda + \log \sum_{\gamma\in\N_0^2}
\mathrm{e}^{\lambda^\gamma - \sum_{\chi\in S} \lambda^{\chi} h^{\chi}_{\beta\gamma}}\,\zeta_\gamma \qquad \forall\,\beta \in \N_0^2.
\]
This equation determines $(\lambda^\beta)_{\beta\in\N_0^2}$ as a function of $\lambda$ and $(\lambda^\chi)_{\chi \in S}$. With this notation, summing out over $\alpha$, we obtain
\[
(\pi_{2,3}\nu)_{\beta\gamma} = (\pi_2\nu)_\beta\,P_{\beta\gamma}
\]
with
\[
(\blacksquare) \qquad P_{\beta\gamma} = \mathrm{e}^{-\lambda-(\lambda^{\beta}-\lambda^\gamma)
- \sum_{\chi\in S}\lambda^{\chi} h^{\chi}_{\beta\gamma}}\,\zeta_{\gamma}.
\]
Summing out also over $\gamma$, we see that 
\[
\sum_{\gamma\in \N_0^2} P_{\beta\gamma} = 1 \qquad \forall\,\beta \in \N_0^2,
\]
i.e., $P = (P_{\beta\gamma})_{\beta,\gamma\in\N_0^2}$ is a Markov transition kernel. Summing out over $\beta$ instead, we get $(\pi_3\nu)_\gamma = \sum_{\beta\in \N_0^2} (\pi_2\nu)_\beta\, P_{\beta\gamma}$. Since $\pi_{1,2}\nu = \pi_{2,3}\nu$ implies that $\pi_3\nu = \pi_2\nu$, the latter says that $\pi_2\nu = \rho$ with $\rho$ the invariant probability law of $P$, i.e.,  
\[
\rho_\gamma = \sum_{\beta\in\N_0^2} \rho_\beta P_{\beta\gamma} \qquad \forall\, \gamma \in \N_0^2. 
\]
Thus, we end up with the minimiser being
\[
\nu_{\alpha\beta\gamma} = \rho_\alpha P_{\alpha\beta} P_{\beta\gamma}
\] 
and the minimum being
\[
\begin{aligned}
J(x,s) &= \sum_{\alpha,\beta,\gamma \in \N_0^2} \nu_{\alpha\beta\gamma} 
\log\left(\frac{\nu_{\alpha\beta\gamma}}{(\pi_{1,2}\nu)_{\alpha\beta}\,\zeta_\gamma}\right)
= \sum_{\alpha,\beta,\gamma \in \N_0^2} \nu_{\alpha\beta\gamma} \log\left(\frac{P_{\beta\gamma}}{\zeta_\gamma}\right)\\
&= \sum_{\alpha,\beta,\gamma \in \N_0^2} \nu_{\alpha\beta\gamma} 
\left(- \lambda - (\lambda^\beta - \lambda^\gamma) - \sum_{\chi \in S} \lambda^\chi h^\chi_{\beta\gamma}\right) 
= - \lambda - \sum_{\chi \in S} \lambda^\chi xs^\chi,  
\end{aligned}
\]
where $(\lambda^\beta)_{\beta\in\N_0^2}$ drops out and the constraints are implemented. The Lagrange multipliers must be chosen such that
\[
(\ast) \qquad 1 = |\nu| = \sum_{\gamma \in \N_0^2} \rho_\gamma,
\qquad 
xs^\chi =  G^\chi(\nu) = \sum_{\beta,\gamma \in \N_0^2} \rho_\beta P_{\beta\gamma} h^\chi_{\beta\gamma}, \quad \chi \in S.
\]
These equations are \emph{coupled}, because $\rho_\gamma$ and $\rho_\beta P_{\beta\gamma}$ depend on $\lambda$ and $\lambda^\chi$, $\chi \in S$, which in turn depend on $(x,s)$. We think of $(\ast)$ as \emph{consistency equations}. Note that the constraint $\pi_{1,2}\nu = \pi_{2,3}\nu$ is already met because $P$ is Markov and $\rho$ is its invariant probability law.

\medskip\noindent
\emph{Step 2.}
Next, we reduce $(\ast)$ by using that not all variables are relevant. To that end, note that $h^\chi_{\beta,\gamma}$ mostly depends on $\gamma$, namely,
\[
h^\chi_{\beta\gamma} = \left\{\begin{array}{ll}
\hat{h}^\chi_\gamma, &\text{if } \beta \in \{0\} \times \N_0,\\
\bar{h}^\chi_\gamma, &\text{if } \beta \in \N \times \N_0,
\end{array}
\right. \chi \in \{-,0\},
\qquad 
h^+_{\beta\gamma} = \hat{h}^+_\gamma = \bar{h}^+_\gamma,
\]
with
\[
\begin{aligned}
\hat{h}^-_\gamma &= 2\,\mathbbm{1}_{\{c>0\}} + \mathbbm{1}_{\{c=0,c'>0\}},\\ 
\bar{h}^-_\gamma &= \mathbbm{1}_{\{c>0\}},\\[0.2cm]
\hat{h}^0_\gamma &= c-\mathbbm{1}_{\{c>0\}}-\mathbbm{1}_{\{c>1\}} + \mathbbm{1}_{\{c=0,\,c^\prime=0\}},\\ 
\bar{h}^0_\gamma &= c-\mathbbm{1}_{\{c>0\}}-\mathbbm{1}_{\{c>1\}},\\[0.2cm]
\hat{h}^+_\gamma &= \bar{h}^+_\gamma = \mathbbm{1}_{\{c>0\}} + \mathbbm{1}_{\{c>1\}}.
\end{aligned}
\]
Since
\[
\hat{h}^-_\gamma - \bar{h}^-_\gamma = \mathbbm{1}_{\{(c,c') \neq (0,0)\}}, \qquad
\hat{h}^0_\gamma - \bar{h}^0_\gamma = \mathbbm{1}_{\{(c,c') = (0,0)\}},
\]
the second part of $(\ast)$ reduces to
\[
xs^\chi = \sum_{\gamma\in\N_0^2} \rho_\gamma \bar{h}^\chi_\gamma 
+ \sum_{b'\in\N_0} \rho_{(0,b')} \times 
\left\{\begin{array}{ll} 
1 - P_{(0,b'),(0,0)}, &\text{if } \chi = -,\\
P_{(0,b'),(0,0)}, &\text{if } \chi = 0,\\
0, &\text{if } \chi = +.
\end{array}
\right.
\]
Since $h^-_{(0,b'),(0,0)} = 0$, $h^0_{(0,b'),(0,0)} = 1$, $h^+_{(0,b'),(0,0)} = 0$, we have from $(\blacksquare)$ that
\[
P_{(0,b'),(0,0)} = C_{b'} \,\ee^{-\lambda - \lambda^0} q^2, 
\qquad C_{b'} = \ee^{-(\lambda^{(0,b')} - \lambda^{(0,0)})},
\]
where we use that $\zeta(0) = 1-p = q$. It follows from $(\square)$ that
\[
\lambda^{(0,b')} = - \lambda + \log \sum_{\gamma \in \N_0^2} 
\ee^{\lambda^\gamma - \sum_{\chi \in S} \lambda^\chi \hat{h}^\chi_\gamma} \zeta_\gamma,
\]
which does not depend on $b'$. Hence, $C_{b'} = 1$ for all $b' \in \N_0$ and so the second sum in the above expression for $xs^\chi$ reduces to   
\[
\left\{
\begin{array}{ll}
\big(1 -\ee^{-\lambda-\lambda^0} q^2\big)\, \rho_0, &\text{if } \chi = -,\\
\big(\ee^{-\lambda-\lambda^0} q^2\big)\,\rho_0, &\text{if } \chi = 0,\\
0, &\text{if } \chi = +, 
\end{array}
\right.
\]
with $\rho_0 = \sum_{b' \in \N_0} \rho_{(0,b')}$. Moreover, since $\bar{h}^-_\gamma$, $\bar{h}^0_\gamma$, $\bar{h}^+_\gamma$ depend on $c$ only, we have 
\[
\sum_{\gamma\in\N_0^2} \rho_\gamma \bar{h}^\chi_\gamma = \sum_{c\in\N_0} \rho_c h^\chi_c, \quad \chi \in S,
\]
with $\rho_c = \sum_{c' \in \N_0} \rho_{(c,c')}$ and
\[
h^-_c = \mathbbm{1}_{\{c>0\}}, \qquad 
h^0_c = c-\mathbbm{1}_{\{c>0\}}-\mathbbm{1}_{\{c>1\}}, \qquad
h^+_c = \mathbbm{1}_{\{c>0\}} + \mathbbm{1}_{\{c>1\}}.
\]
Thus, the second part of $(\ast)$ further reduces to 
\[
xs^\chi = \rho_0 A^\chi_0 + \sum_{c\in\N} \rho_c h^\chi_c, \quad \chi \in S, \qquad 
A^\chi_0 = \left\{\begin{array}{ll}
1-\ee^{-\lambda-\lambda^0} q^2, &\text{if } \chi = -,\\
\ee^{-\lambda-\lambda^0} q^2, &\text{if } \chi = 0,\\
0, &\text{if } \chi = +, 
\end{array}
\right.
\]
which only involves the probability measure $\rho=(\rho_c)_{c \in \N_0}$ and the Lagrange multipliers $\lambda,\lambda^0$. Next, use $(\square)$ and $(\blacksquare)$ to see that $\lambda^{(b,b')}$ and $P_{(b,b'),(c,c')}$ do not depend on $b'$ because $h^\chi_{(b,b'),(c,c')}$, $\chi \in S$ does not. Therefore, summing the invariance relation over $c'$, we get
\[
\rho_c = \sum_{b\in\N_0} \rho_b P_{bc}
\] 
with 
\[
P_{bc} = \ee^{-\lambda-(\lambda_*^b-\lambda_*^c)} \zeta(c)
\sum_{c' \in \N_0} \ee^{- \sum_{\chi \in S} \lambda^\chi h^\chi_{(b,\cdot),(c,c')}}\,\zeta(c'),
\]
where $\lambda_*^b = \lambda^{(b,b')}$. Using the formulas for $h^\chi_{(b,\cdot),(c,c')}$, we get
\[
P_{bc} = \ee^{-\lambda-(\lambda_*^b-\lambda_*^c)} Q_c \times
\left\{\begin{array}{ll}
1, &\text{if } b>0,\\
R_c, &\text{if } b=0,\\ 
\end{array}
\right.
\qquad Q_c = \ee^{-\sum_{\chi \in S} \lambda^\chi h^\chi_c} \zeta(c), 
\]
with the abbreviation
\[
R_c = \left\{\begin{array}{ll}
\ee^{-\lambda^-}, &\text{if } c>0,\\
p\,\ee^{-\lambda^-} + q\,\ee^{-\lambda^0}, &\text{if } c=0.
\end{array}
\right.
\]
The latter arrises from the sum
\[
R_c = \sum_{c'\in\N_0} \zeta(c')\,\ee^{-\lambda^- \mathbbm{1}_{\{(c,c') \neq (0,0)\}}
-\lambda^0 \mathbbm{1}_{\{(c,c') = (0,0)\}}}.
\]

\medskip\noindent
\emph{Step 3.} 
Next, we do a reparametrisation. Put
\[
u = \ee^{-\lambda}, \qquad u^\chi = \ee^{-\lambda^\chi}, \qquad u_*^b = \ee^{-\lambda_*^b},
\]
which are elements of $(0,\infty)$, to write
\[
P_{bc} = \frac{u u_*^b}{u_*^c}\,Q_c \times
\left\{\begin{array}{ll}
1, &\text{if } b>0,\\
R_c, &\text{if } b=0,\\ 
\end{array}
\right.
\qquad Q_c = q (1-q)^c \prod_{\chi \in S} (u^\chi)^{h^\chi_c}, 
\]
where we use that $\zeta(c) = q(1-q)^c$, $c \in \N_0$. Using once more the formulas for $h^\chi_{(b,\cdot),(c,c')}$, in combination with $\sum_{c\in\N_0} P_{bc} =1$ for all $b\in\N_0$, we get 
\[
\frac{1}{u_*^b} = \left\{\begin{array}{ll}
u \sum_{c\in\N_0} \frac{1}{u_*^c}\,Q_c, &\text{if } b>0,\\[0.4cm]
u \sum_{c\in\N_0} \frac{1}{u_*^c}\, R_c\,Q_c , &\text{if } b=0.
\end{array}
\right. 
\]
It follows that $u_*^b$ does not depend on $b$ when $b>0$. Let $v$ denote the quotient of its value for $b=0$ and its value for $b>0$. Then $v$ solves the two equations
\[
(\triangle) \qquad \frac{1}{v}\,Q_0 + \sum_{c>0} Q_c = \frac{1}{u} = R_0\,Q_0 + v \sum_{c>0} R_c\,Q_c. 
\]
Moreover,
\[
P_{bc} = u \times \left\{\begin{array}{ll}
Q_c, &b,c>0,\\[0.2cm]
v\, R_c\,Q_c, &b=0,c>0,\\[0.2cm]
\frac{1}{v}\, Q_0, &b>0,c=0,\\[0.2cm]
R_0\,Q_0, &b=c=0.
\end{array}
\right.
\]
Inserting the formulas for $h^\chi_c$, $\chi \in S$, we have
\[
Q_0 = q, \qquad Q_1 = pq u^-u^+,
\qquad Q_c = q\, [p u^0]^c\, u^- \left(\frac{u^+}{u^0}\right)^2, \quad c > 1.
\]
Put $W = \sum_{c>0} Q_c$. Then
\[
(\blacktriangle_1) \qquad W = pq u^-u^+ \left[1 + \frac{pu^+}{1-p u^0}\right],
\]
and $(\triangle)$ becomes
\[
\frac{q}{v} + W = \frac{1}{u} = qR_0 + v u^-W 
\]
with $R_0 = pu^-+qu^0$. From this we find not only that
\[
\frac{1}{v} = \frac{1}{q}\left(\frac{1}{u}-W\right), \qquad v = \frac{1}{u^-\,W}\left(\frac{1}{u} - q R_0\right),
\]
but also that 
\[
(\blacktriangle_2) \qquad W = \frac{1- q R_0 u}{u[qu^-u + (1-q R_0 u)]}.
\]
Note that $(\blacktriangle_1) = (\blacktriangle_2)$ gives a relation between $u$ and $(u^\chi)_{\chi \in S}$.

\medskip\noindent
\emph{Step 4.}
The above allows us to make the consistency equations explicit. The invariance relation gives
\[
\begin{aligned}
\rho_0 &= \rho_0\,q R_0 u + (1-\rho_0)(1-Wu),\\
\rho_c &= Q_c\,\left[\rho_0\,\frac{1}{W}(1-q R_0 u) + (1-\rho_0) u\right], \quad c>0,
\end{aligned}
\] 
where we use that $\sum_{b\in\N_0} \rho_b =1$. The first equation gives a formula for $\rho_0$:
\[
\rho_0 = \frac{(1-Wu)}{(1-Wu) + (1-q R_0 u)}, \qquad \rho_c = \frac{Q_c}{W}\,(1-\rho_0), \quad c>0.
\]
These formulas allow us to express the constraints in $(\ast)$ in terms of $\rho_0$. Indeed, the second part of $(\ast)$ reads
\[
(\ast\ast) \qquad xs^\chi = \rho_0 A_0^\chi + (1-\rho_0) \Sigma^\chi, \qquad \chi \in S,
\]
with
\[
A_0^\chi = \left\{\begin{array}{ll}
1- u u^0 q^2, &\text{if } \chi = -,\\
u u^0 q^2, &\text{if } \chi = 0,\\
0, &\text{if } \chi = +, 
\end{array}
\right.
\]
where we require that $0 < u u^0 q^2 \leq 1$ (because $A^-_0$ and $A^0_0$ are probabilities), and
\[
\begin{aligned}
\Sigma^- &= (1-\rho_0)^{-1} \sum_{c>0} \rho_c = 1,\\ 
\Sigma^0 &= (1-\rho_0)^{-1} \sum_{c>1} \rho_c (c-2) =
\frac{pu^0}{1-pu^0}\,\frac{pu^+}{1-pu^0+pu^+},\\
\Sigma^+ &= (1-\rho_0)^{-1} \left[\sum_{c>0} \rho_c + \sum_{c>1} \rho_c\right] 
= 2-\frac{1-pu^0}{1-pu^0+pu^+},
\end{aligned}
\]
where in the last two lines we use $(\blacktriangle_1)$, and we require that $p u^0 < 1$ to ensure that the geometric sums used in the calculation converge. The three equations in $(\ast\ast)$ provide three constraints on the Lagrange multipliers in terms of $(xs^\chi)_{\chi \in S}$, where we note that $(\Sigma^\chi)_{\chi \in S}$ does not depend on $u,u^-$. The first part of $(\ast)$ has already been taken care of, but there is a fourth constraint coming from $(\blacktriangle_1) = (\blacktriangle_2)$. After inserting the formula for $\rho_0$, we obtain four explicit equations from which $u$ and $u^\chi$, $\chi \in S$, must be solved for given $(x,s)$. The solution determines the rate function as
\[
(\heartsuit) \qquad J(x,s) = \log u + x \sum_{\chi \in S} s^\chi \log u^\chi,
\]  
which is valid for $(x,s) \in \mathrm{int}(\cD)$.

\medskip\noindent
\emph{Step 5.}
We are finally ready to compute the rate function. Abbreviate
\[
A = pu^-+qu^0, \qquad B = q(u^--u^0), \qquad C = pu^-u^+ \left(1+\frac{pu^+}{1-pu^0}\right).
\]
The above expressions yield
\[
(\diamondsuit) \qquad \rho_0 = \frac{qu^-u}{qu^-u\,(2-quA)+(1-quA)^2}.
\] 
The relation $(\blacktriangle_1) = (\blacktriangle_2)$ leads to a rewrite of the fourth constraint as
\[
\frac{(1-quA)}{u[(1-quA)+qu^-u]} = qC.
\]
This gives the quadratic equation $0 = BC(qu)^2 + (A+C)qu-1$, whose solution is
\[
(\blacktriangle) \qquad  qu = \frac{-(A+C) + \sqrt{(A+C)^2+4BC}}{2BC}, \quad B \neq 0,  \qquad qu = \frac{1}{A+C}, \quad B=0.
\]
(We must take the positive root because $A,C \in (0,\infty)$.) The formula in $(\blacktriangle)$ expresses $u$ in terms of $(u^-,u^0,u^+)$, so that $(\ast\ast)$ provides three equations from which the three unknowns $(u^-,u^0,u^+)$ must be solved.

\medskip\noindent
(c) To show that $(x,s) \mapsto J(x,s)$ is analytic on $\mathrm{int}(\cD)$, we use the implicit function theorem \cite[Chapter II, Theorem 1.5]{B}. Note that $\mathrm{int}(\cD)$ is the set of $(x,s)$-values for which $(\ast\ast)$ has a solution in $(0,\infty)^3$. Since $(\ast\ast)$ is a set of three equations for three unknowns, it has a unique solution in the effective domain of $J$ as long as the three equations are non-redundant. But they clearly are, e.g.\ because $A_0^-,A_0^0$ depend on $uu^0$ while $A_0^+$ does not, $\Sigma^0,\Sigma^+$ depend on $u^0,u^+$ while $\Sigma^-$ does not, and $u$ depends on $u^-$. Since, in $\mathrm{int}(\cD)$, $(x,s)$ is analytic as a function of $(u^-,u^0,u^+)$, it follows that $(u^-,u^0,u^+)$ is analytic as a function of $(x,s)$. The claim is therefore immediate from $(\heartsuit)$.  

\medskip\noindent
By analyticity, $(x,s) \mapsto J(x,s)$ is strictly convex on $\mathrm{int}(\cD)$ unless it is everywhere linear on $\mathrm{int}(\cD)$. But this is excluded by the fact that it is non-negative and has a unique zero in the interior of its effective domain.     

\medskip\noindent
(d) Certain constraints on $(x,s)$ are needed to guarantee the existence of a solution of $(\ast\ast)$. Clearly, $1 < x < \infty$ (see Remark~\ref{Jx} below). As shown in Figs.~\ref{fig:plot1}--\ref{fig:plot2}, for fixed $x$ there are four linear constraints. On the interior of the simplex we have $s^-+s^0<1$, which is the first constraint. Moreover, $xs^- = 1-A^0_0\rho_0 \in (0,1)$ and $x(s^-+s^0) = 1+(1-\rho_0)\Sigma^0$, and so we also require that $s^- < 1/x$ and $s^-+s^0>1/x$, which are the second and the third constraint. Both constraints can be reached arbitrarily closely by letting $u^- \downarrow 0$ (so that $\rho_0 \downarrow 0$), respectively, $u^0 \downarrow 0$ (so that $\Sigma^0 \downarrow 0$). To get the fourth constraint we compute $3s^-+s^0$. Summing $(\ast\ast)$ over $\chi \in S$ and using that $A^-_0 + A^0_0 + A^+_0 = 1$, we get
\[
x = \rho_0 + (1-\rho_0)(\Sigma^- + \Sigma^0 + \Sigma^+).
\]
Using that $A^-_0 + A^0_0 = 1$ and $\Sigma^- = 1$ this gives
\[
3s^- + s^0 = \frac{\rho_0 (1+2A^-_0) + (1-\rho_0)(3+\Sigma^0)}{1+ (1-\rho_0)(\Sigma^0 + \Sigma^+)}.
\] 
In order for this to be $>1$, we require that 
\[
\rho_0 (1+2A^-_0) + (1-\rho_0) 3 > 1+ (1-\rho_0)\Sigma^+,
\]
which is the same as
\[
\rho_0 2A^-_0 + (1-\rho_0)(2-\Sigma^+) > 0.
\] 
But the latter is trivially satisfied, because $\Sigma^+ < 2$ and because $\rho_0 = 1$ implies $quA = 1$, i.e., $uu^0q^2 + uu^-pq=1$, and hence $A^-_0 = 1-uu^0q^2 = uu^-pq>0$. The constraint can be reached arbitrarily closely by letting $u^- \downarrow 0$ (so that $\rho_0 \downarrow 0$) and $u^+ \uparrow \infty$ (so that $\Sigma^+ \uparrow 2$). 
\end{proof}

\begin{remark}
\label{rem:lsc}
{\rm (1) We can use the above calculation to verify that the function $\theta \mapsto \bar{I}(\theta)$ introduced in the proof of Lemma~\ref{lem:LDP-type-fractions}(b) is lower semi-continuous. To see how, note that the Lagrange multiplier analysis provides a description of $\theta \mapsto \bar{I}(\theta)$ through $(x,s) \mapsto J(x,s)=\bar{I}(xs)$, and that $J$ is analytic on $\mathrm{int}(\cD)$, and hence continuous on $\mathrm{int}(\cD)$. Since $J$ is infinite on $\cD^c$, it suffices to verify that $J$ is lower semi-continuous at $\partial\cD$.\\ 
(2) The lower semi-continuity at the boundary can be verified by considering the different boundary regimes of the parameters $u,u^-,u^0,u^+$ case by case. Not all combinations are possible. In each case, we can check from $(\heartsuit)$ that the corresponding value of $J$ converges either to a finite limit or to infinity. For instance, if $u^-$ and $u^0$ are kept fixed while $u^+ \to \infty$, then, by $(\blacktriangle)$, $u \asymp (u^+)^{-2}$. Moreover, $\rho_0 \downarrow 0$, $\Sigma^- = 1$, while $\Sigma^0$ converges to a finite limit, and $\Sigma^+ \uparrow 2$. Hence, by $(**)$,
\[
xs^-\to 1, \qquad xs^0 \to \Sigma^0, \qquad xs^+\to 2.
\]
By $(\heartsuit)$, we have
\[
J(x,s) = \log u + xs^-\log u^- + xs^0\log u^0 + xs^+ \log u^+.
\]
Since $\log u+2\log u^+$ converges to a finite limit and $xs^+ -2 = O((u^+)^{-1})$, it follows that $\log u + xs^+\log u^+$ converges to a finite limit. Moreover, $u^-$ and $u^0$ are fixed, while $xs^- \to 1$ and $xs^0$ converges to a finite limit. Hence the remaining terms stay bounded, and therefore $J(x,s)$ converges to a finite limit.\\
(3) The other boundary regimes can be analysed similarly. Consequently, whenever $(x,s)$ approaches a boundary point of $\cD$ from within $\operatorname{int}(\cD)$, the value of $J(x,s)$ converges either to a finite limit or to infinity. Since $J$ is convex and finite on $\operatorname{int}(\cD)$, this boundary convergence implies that $J$ is lower semi-continuous at $\partial\cD$. Since the map $(x,s)\mapsto xs$ is a homeomorphism from $(0,\infty) \times \mathcal{S}_2$ onto $F_0$, the lower semi-continuity of $J$ on $(0,\infty) \times \mathcal{S}_2$ implies the lower semi-continuity of $\bar I$ on $F_0$. Since $\bar{I} = \infty$ on $\R^3\setminus F_0$, the above boundary analysis shows that $\bar{I}$ is lower semi-continuous on all of $\R^3$.}
\end{remark}

\begin{remark}
\label{rem:numcomp}
{\rm Even though it is difficult to invert $(\ast\ast)$ to compute $(u^\chi)_{\chi \in S}$ as a function of $(x,(s^\chi)_{\chi \in S})$ (recall that $\sum_{\chi \in S} s^\chi = 1$), it is easy to compute $(x,(s^\chi)_{\chi \in S})$ as a function on $(u^\chi)_{\chi \in S}$. Hence $(\ast\ast)$ provides a way to \emph{numerically} compute $J(x,s)$ for a range of values $(x,s)$, simply by letting $(u^\chi)_{\chi \in S}$ run through $(0,\infty)^3$, and for each choice computing $(x,s)$ via $(\ast\ast)$ and $J(x,s)$ via $(\heartsuit)$. This explains why the above computation is helpful to \emph{quantify} the rate function. See Figs.~\ref{fig:plot1}--\ref{fig:plot2} for an illustration.}
\end{remark}

\begin{remark}
{\rm Note that the choice $u=1$ and $u^\chi=1$, $\chi \in S$ (i.e., all Lagrange multipliers are zero) yields
\[
A = 1, \qquad B = 0, \qquad C = \frac{p}{q},
\]
and hence
\[
\Sigma^- = 1, \qquad \Sigma^0 = \frac{p^2}{q}, \qquad \Sigma^+ = 1+p, \qquad
A_0^\chi = \left\{\begin{array}{ll}
1- q^2, &\text{if } \chi = -,\\
q^2, &\text{if } \chi = 0,\\
0, &\text{if } \chi = +, 
\end{array}
\right.
\qquad \rho_0 = q.
\]
This corresponds to 
\[
xs^- = q(1-q^2) + p, \qquad xs^0 = q^3 + \frac{p^3}{q}, \qquad xs^+ = p(1+p),
\]
which is the unique minimiser $(x^*,s^*)$.}
\end{remark}  

\begin{remark}
\label{Jx}
{\rm It is easy to compute $J(x) = \inf_{s \in S} J(x,s)$, which is the rate function of the sequence $\{N_j/j\}_{j \in \N_4}$ (recall Section~\ref{sec:LDP}). Indeed, this amounts to choosing $\lambda^\chi = \bar\lambda$, $\chi \in S$, in the Lagrangian in Step 1, which means that the last term in the Lagrangian becomes $\bar\lambda(\|G(\nu)\|_1-x)$. Putting $\bar{u} = \ee^{-\bar\lambda}$, and summing $(\ast\ast)$ over $\chi \in S$, we find that $x = 1 + (1-\rho_0) \frac{1}{1-p\bar{u}}$, where we use that $\sum_{\chi \in S} A_0^\chi = 1$ and $\sum_{\chi \in S} \Sigma^\chi = 2 + \frac{p\bar{u}}{1-p\bar{u}}$. Since $A=\bar{u}$, $B=0$, $C = \bar{u}\,\frac{p\bar{u}}{1-p\bar{u}}$, we get from $(\diamondsuit)$ that $\rho_0 = qu\bar{u}$, and from $(\blacktriangle)$ that $qu\bar{u} = 1-p\bar{u}$. Hence $x = \frac{1}{1-p\bar{u}}$, and so $\bar{u} = \frac{x-1}{px}$. Moreover, $u = \frac{1-p\bar{u}}{q\bar{u}} = \frac{p}{q(x-1)}$. Consequently,
\[
J(x) = \log\left(\frac{p}{q(x-1)}\right) + x \log\left(\frac{x-1}{px}\right)
= (x-1) \log\left(\frac{x-1}{px}\right) + \log\left(\frac{1}{qx}\right), \qquad x \in (1,\infty).
\]   
Note that $J'(x) = \log(\frac{x-1}{px})$ and $J''(x) = \frac{1}{x(x-1)}$. Hence, $J$ has a unique zero at $x_* = \frac{1}{q}$, and is strictly convex. Moreover, $\lim_{x \downarrow 1} J(x) = \log(1/q)$.}
\end{remark}



\newpage

\begin{figure}[htbp]
\hspace{1cm}
\includegraphics[scale=0.25]{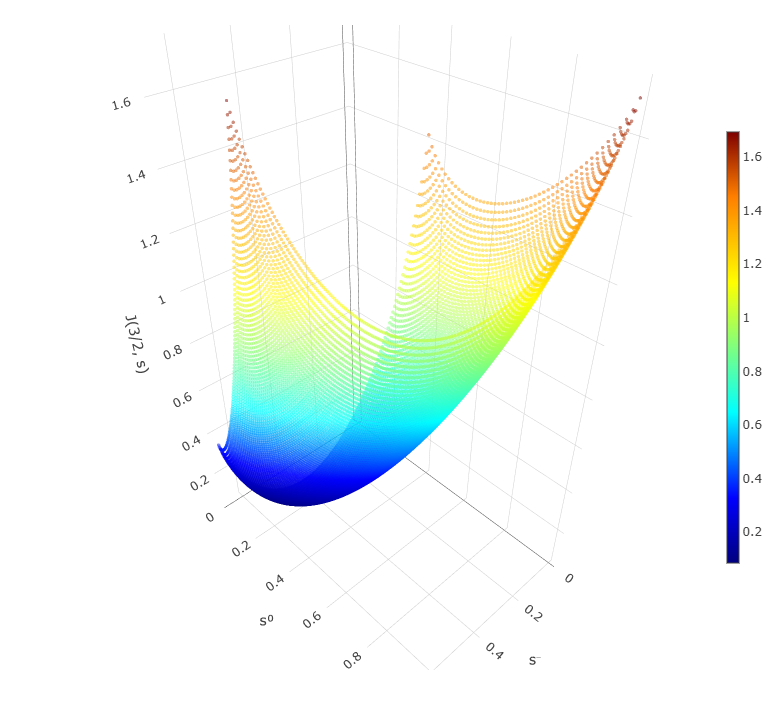}
\hspace{1cm}
\includegraphics[scale=0.4]{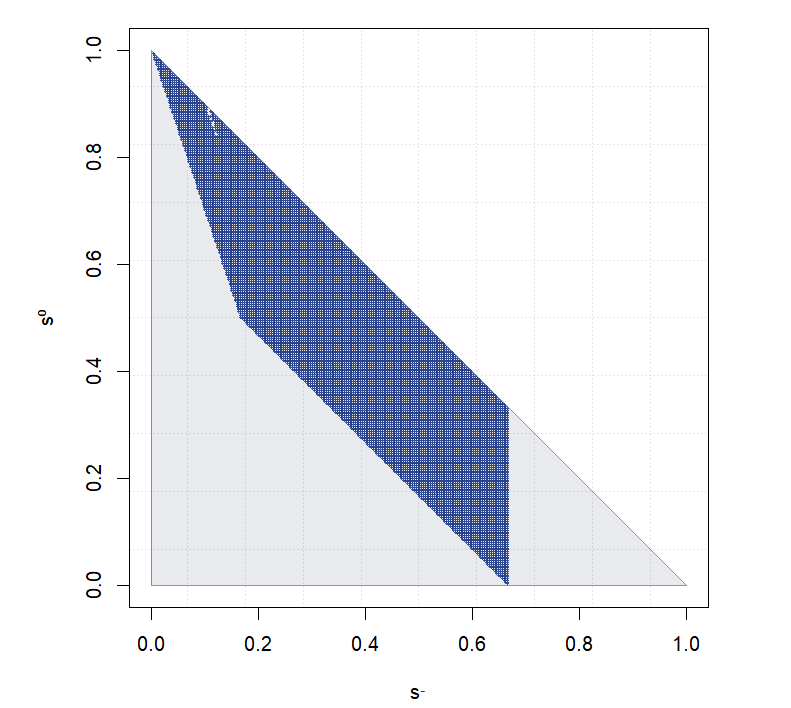}

\hspace{1cm}
\includegraphics[scale=0.25]{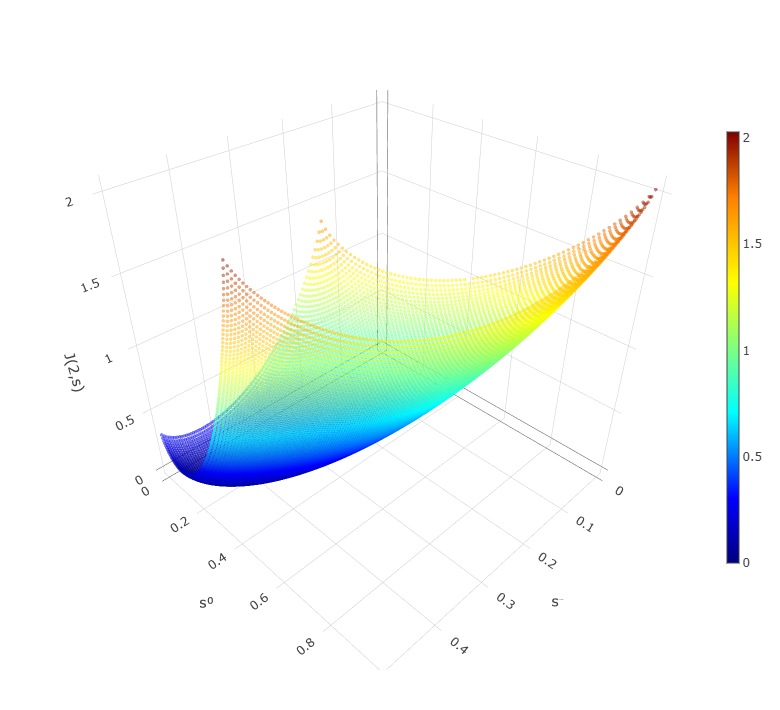}
\hspace{1cm}
\includegraphics[scale=0.4]{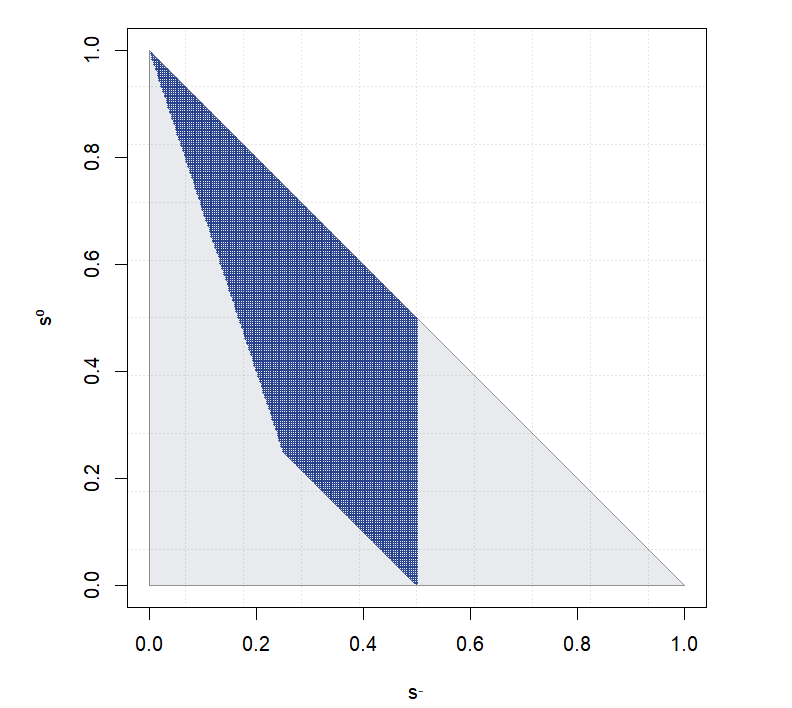}

\hspace{1cm}
\includegraphics[scale=0.25]{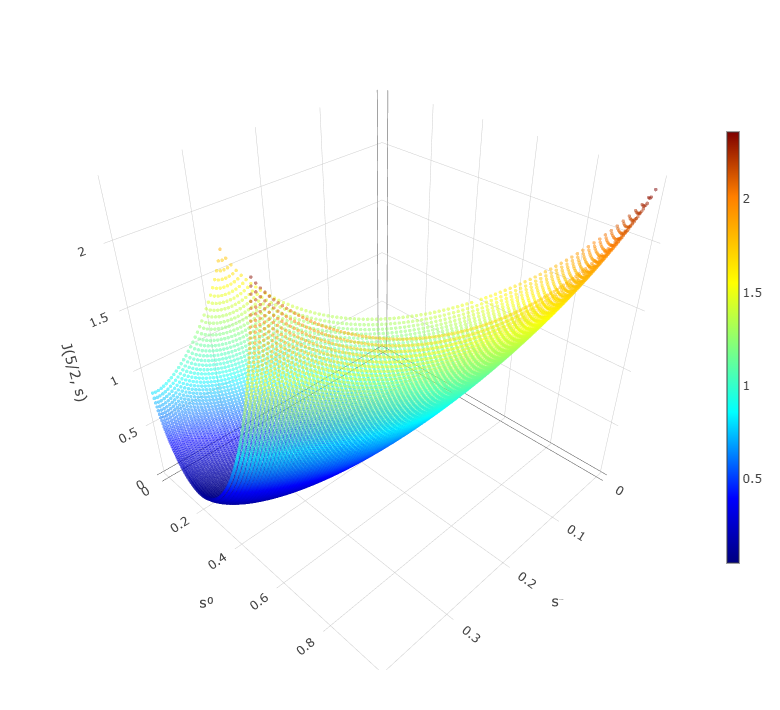}
\hspace{1cm}
\includegraphics[scale=0.4]{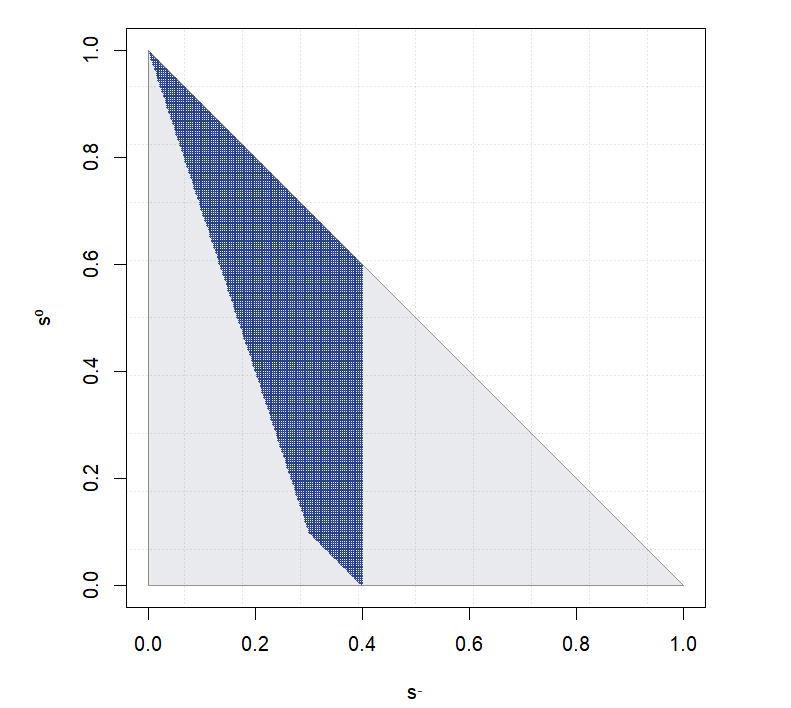}
\caption{\small Numerical plots for $p=\tfrac12$ and $x=\tfrac32$, $x=2$, $x=\tfrac52$, respectively. \emph{Left}: Plot of $s \mapsto J(x,s)$. The height represents the value of $J(x,s)$, the plane represents the simplex $s=(s^-,s^0,1-s^--s^0)$. \emph{Right}: The dotted region is the effective domain of $s \mapsto J(x,s)$ in the simplex. \emph{Comments}: (1) Since $x^*=2$ when $p=\tfrac12$, the middle left figure achieves a unique zero at $s=(s^-,s^0,s^+) = (\tfrac{7}{16},\tfrac{3}{16},\tfrac{6}{16})$; (2) The upper cut-offs $s^-+s^0 < 1$, $s^-<1/x$ and the lower cut-offs $s^-+s^0>1/x$, $3s^-+s^0 > 1$ are clearly visible in the right figures.}
\label{fig:plot1}
\end{figure}

\newpage

\begin{figure}[htbp]
\hspace{1cm}
\includegraphics[scale=0.25]{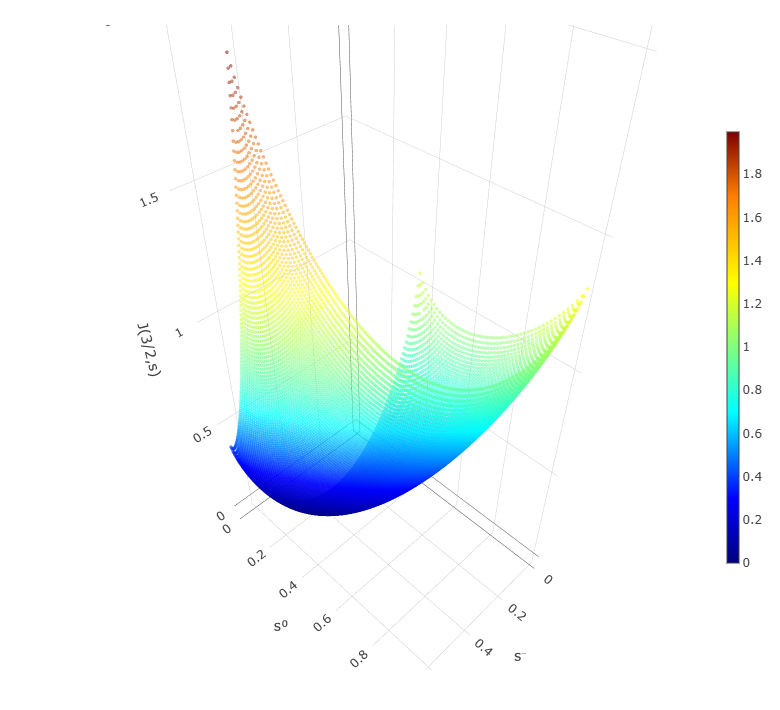}
\hspace{0.9cm}
\includegraphics[scale=0.42]{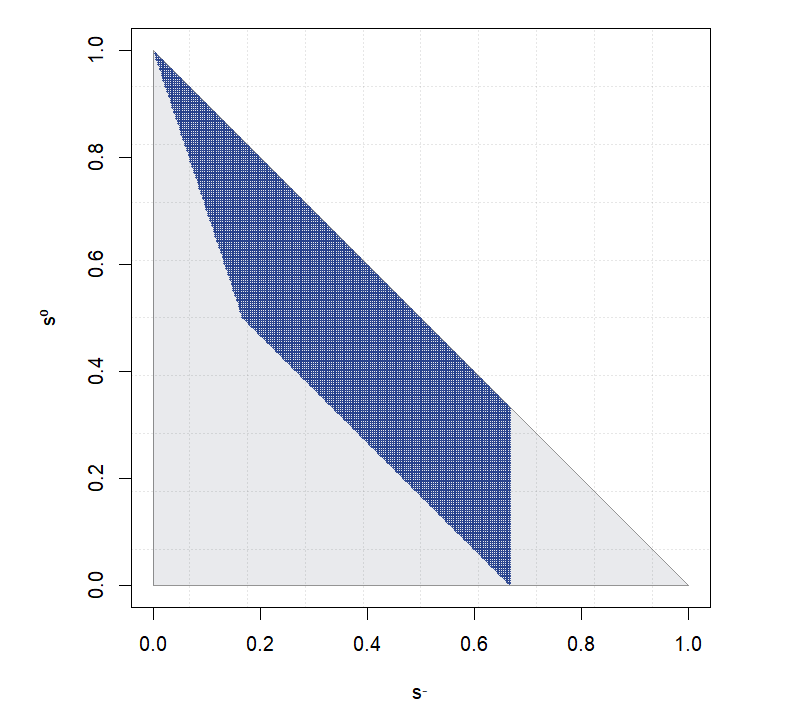}

\hspace{1cm}
\includegraphics[scale=0.25]{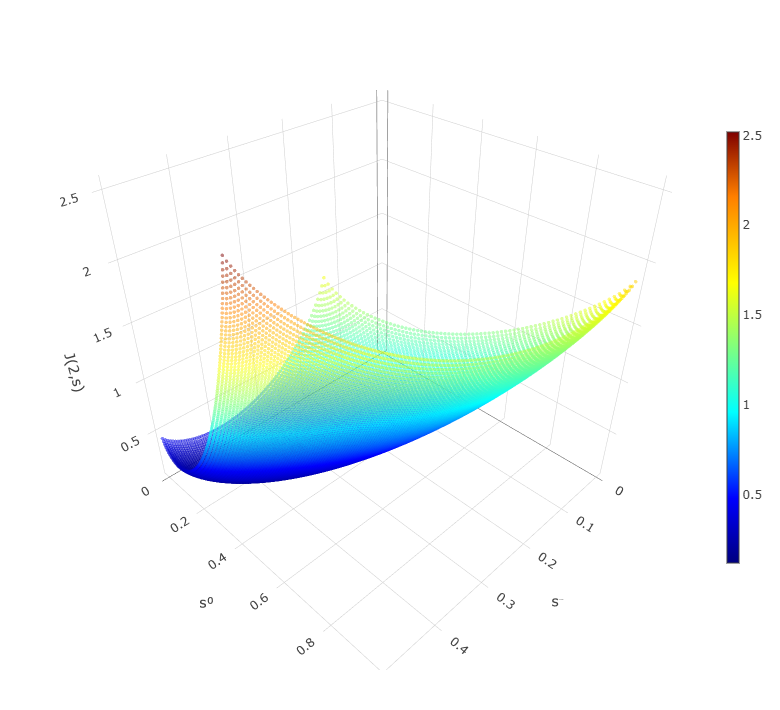}
\hspace{0.85cm}
\includegraphics[scale=0.42]{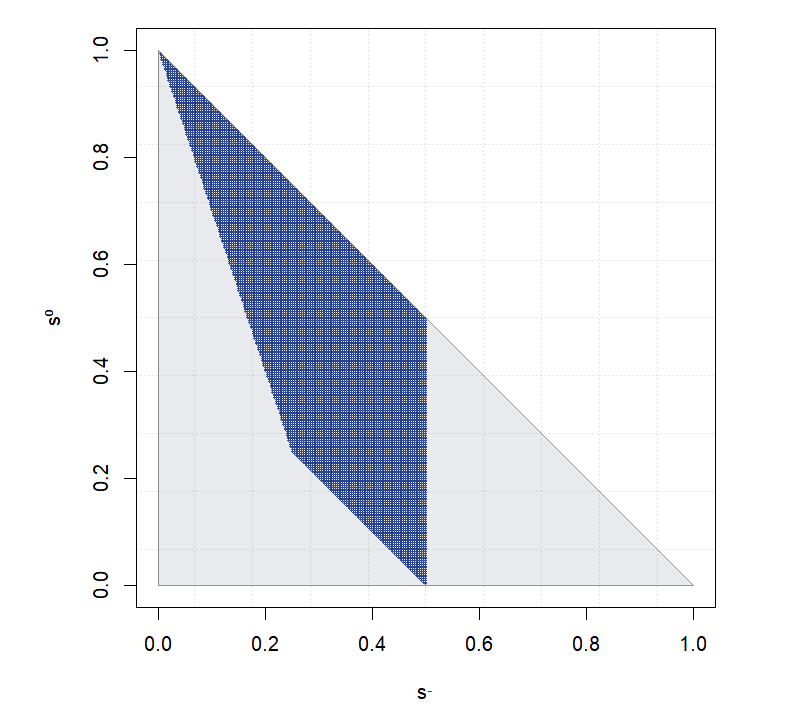}

\hspace{1cm}
\includegraphics[scale=0.25]{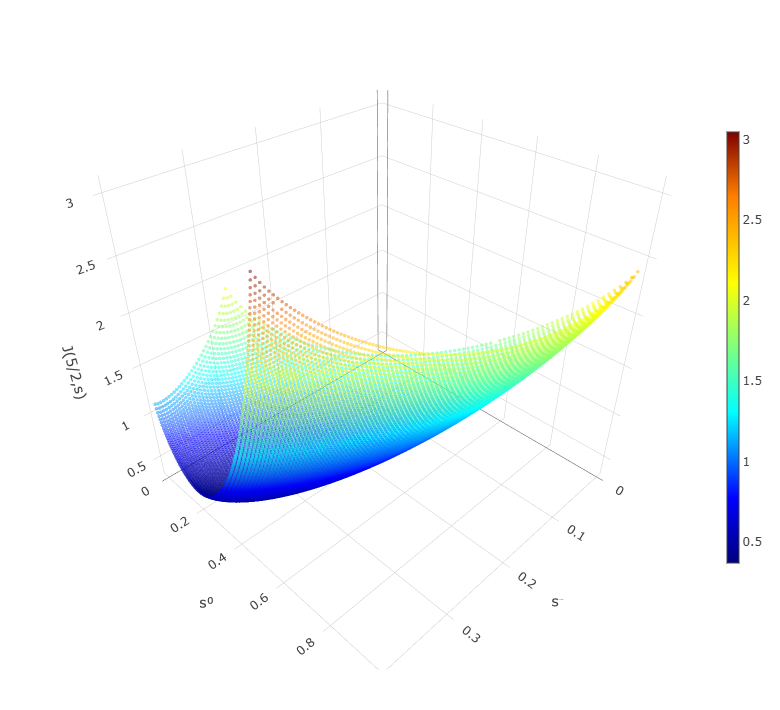}
\hspace{0.9cm}
\includegraphics[scale=0.42]{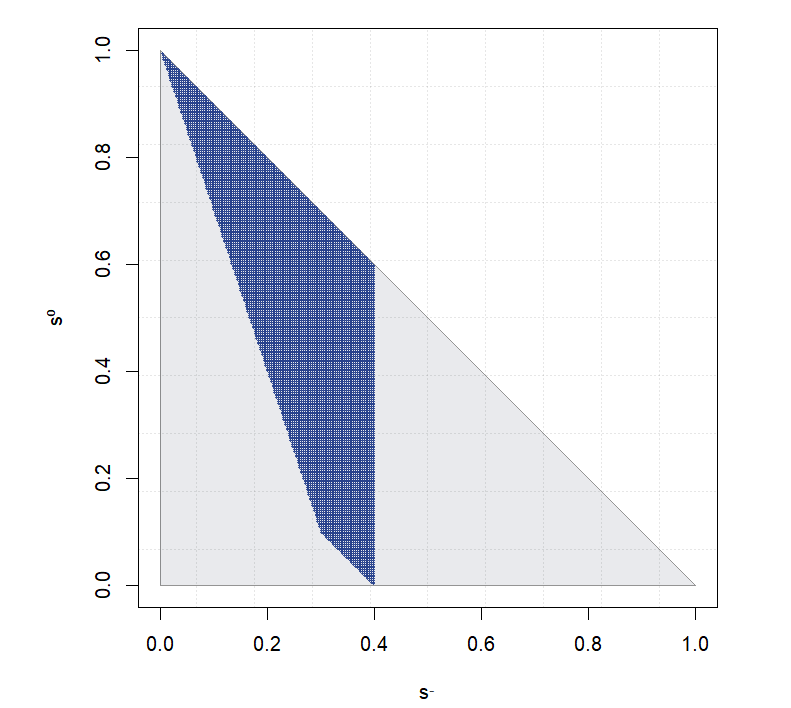}
\caption{\small Same as in Fig.~\ref{fig:plot1}, but with $p=\tfrac13$. Note that the right figures are the same as in Fig.~\ref{fig:plot1}, i.e., the support does not depend on $p$.}
\label{fig:plot2}
\end{figure}


\end{document}